\documentclass[11pt,reqno]{amsart}
\usepackage{graphicx}
\usepackage{color}
\usepackage{framed}
\usepackage{amsmath,amssymb,amsthm,amsfonts,mathrsfs,stmaryrd}
\usepackage{cases,indentfirst}
\usepackage{multicol}
\usepackage[colorlinks=true, linkcolor=black,filecolor=black,      
    urlcolor=black,citecolor=black]{hyperref}
\makeatletter

\newcommand{\Rmnum}[1]{\expandafter\@slowromancap\romannumeral #1@}
\makeatother
\newtheorem{theorem}{Theorem}[section]

\newtheorem{lemma}{Lemma}[section]

\newtheorem{proposition}{Proposition}[section]
\newtheorem{remark}{Remark}[section]
\usepackage[numbers,sort&compress]{natbib}
\theoremstyle{definition}
\numberwithin{equation}{section}

\allowdisplaybreaks[4]
\begin{document}
\author[G.-Y. Hong]{Guangyi Hong\textsuperscript{$ \ast $}}\thanks{$^{\ast}$Corresponding author.}\address{Guangyi Hong\newline\indent Department of Applied Mathematics\newline\indent Hong Kong Polytechnic University\newline\indent Hung Hom,
Kowloon, Hong Kong, P. R. China}
\email{gyhmath05@outlook.com}

\author[Z.-A. Wang]{Zhi-an Wang}\address{Zhi-an Wang\newline\indent Department of Applied Mathematics\newline\indent Hong Kong Polytechnic University\newline\indent Hung Hom,
Kowloon, Hong Kong,  P. R. China}\email{mawza@polyu.edu.hk}


\title[exogenous chemotaxis systems with physical boundary conditions]{\bf Asymptotic stability of exogenous chemotaxis systems with physical boundary conditions}
\begin{abstract}
In this paper, we consider the exogenous chemotaxis system with physical mixed zero-flux and Dirichlet boundary conditions in one dimension. Since the Dirichlet boundary condition can not contribute necessary estimates for the cross-diffusion structure in the system, the global-in-time existence and asymptotic behavior of solutions remain open up to date. In this paper, we overcome this difficulty by employing the technique of taking anti-derivative so that the Dirichlet boundary condition can be fully used, and show that the system admits global strong solutions which exponentially stabilize to the unique stationary solution as time tends to infinity against some suitable small perturbations.  To the best of our knowledge, this is the first result obtained on the global well-posedness and asymptotic behavior of solutions to the exogenous chemotaxis system with physical boundary conditions.



\vspace*{4mm}

\vspace*{1mm}

\end{abstract}
\maketitle

 \section{Introduction} 
 \label{sec:introduction}
Chemotaxis, the directional movement of cells in response to a chemical stimulus gradient, is important for bacteria to find food (e.g., glucose) or to flee from poisons  \cite{wadhams2004making} and critical to early development, normal function such as wound healing/inflammation and pathological process like cancer metastasis \cite{wang2011signaling}. Mathematical models of chemotaxis were firstly developed by Keller-Segel in 1970s with two prototypes describing endogenous and exogenous chemotaxis, respectively. In the endogenous chemotaxis, cells respond to a chemical signal that is released from cells themselves. While in the exogenous chemotaxis, cells respond to an external chemical signal (such as oxygen, light or food). The typical example of endogenous chemotaxis is the spontaneous aggregation of {\it Dictyostelium discoideum} (Dd) cells in response to the chemical cyclic adenosine monophosphate (cAMP) secreted by Dd cells \cite{Bonner}, which was first modeled mathematically by Keller and Segel in \cite{keller1970initiation-produce}. For such aggregation Keller-Segel models, the homogeneous Neumann boundary conditions are usually prescribed to reproduce the aggregating patterns \cite{horstmann20031970, Hillen-Painter-JMB}. The prominent example of exogenous chemotaxis was reported in \cite{Adler66} where motile {\it Escherichia coli} placed at one end of a capillary tube containing an energy source and oxygen migrate out into the tube in the form of traveling bands clearly visible to the naked eye. The mathematical model was subsequently proposed by Keller and Segel in \cite{keller1971traveling}, which reads as
\begin{gather}\label{multi-D-model}
 \displaystyle \begin{cases}
 \displaystyle  u _{t}=\Delta u - \nabla \cdot  \left(u\nabla \phi(v) \right)&\mbox{in}\ \ \Omega,\\[1mm]
\displaystyle v _{t}=D \Delta v - uv^m &\mbox{in}\ \ \Omega,
 \end{cases}
\end{gather}
where $u$ and $v$ denote the bacterial density and oxygen concentration, respectively, at position $x \in \Omega$ and time $t>0$. $D>0$ and $m>0$ account for the chemical diffusivity and consumption rate, respectively, and $\phi(v)$ is called the chemotactic sensitivity function which typically has two prototypes: $\phi(v)=\ln v$ (logarithmic sensitivity) and $\phi(v)=v$ (linear sensitivity). The logarithmic sensitivity was originally used in \cite{keller1971traveling} based on the Weber-Fechner law (the sensory response to a stimulus is logarithmic) which has various biological applications (cf. \cite{kalinin2009logarithmic, dehaene2003neural,LSN-Mathbio}). It was mentioned in \cite[p.241]{keller1971traveling} that the oxygen diffusion rate $D$ is negligible (i.e. $0<D\ll1$) compared to the bacterial diffusion rate. The existence of traveling wave solutions to \eqref{multi-D-model} with logarithmic sensitivity with $D\geq 0$ was shown in \cite{KellerOdell75ms, Schw1, wang2013mathematics} for any $0\leq m\leq 1$, while the stability of traveling wavefronts for the case $m=1$ was obtained in \cite{LW091,LW1,LLW,peng2018nonlinear, ChoiK1,ChoiK2} and the instability of pulsating wave for the case $m=0$ was investigated in \cite{davis2017absolute, nagai1991traveling}.

When considering the exogenous chemotaxis system \eqref{multi-D-model} in a bounded domain $\Omega$, the relevant physical boundary conditions (for instance see the experiment in \cite{Adler66}) are
 \begin{gather} \label{bd-conditions}
\displaystyle  \partial_\nu u-u \partial_\nu v =0,\ \ v =v _{\ast}\ \ \mbox{on}\ \ \partial \Omega,
\end{gather}
where $\partial_\nu=\frac{\partial}{\partial \nu}$ is the normal derivative on the boundary with $\nu$ denoting the outward unit normal vector of $\partial \Omega$, and the constant $v_*>0$ denotes the boundary value of $v$. That is, the zero-flux boundary condition and Dirichlet boundary condition are imposed to cell density $u$ and chemical concentration $v$, respectively. The Keller-Segel system \eqref{multi-D-model} subject to the boundary condition \eqref{bd-conditions} has also been used in the chemotaxis-fluid model in \cite{tuval2005bacterial} to describe the boundary accumulation layer of aerobic bacterial chemotaxis towards the drop edge (air-water interface) in a sessile drop mixed with {\it Bacillus subtilis} bacteria. The model in \cite{tuval2005bacterial} reads
\begin{gather}\label{fluid-chemo}
\displaystyle \begin{cases}
\displaystyle  u _{t}+ {\bf w} \cdot \nabla u= \Delta u- \nabla \cdot \left( u \nabla v \right),\\[1mm]
  \displaystyle v _{t}+{\bf w} \cdot \nabla v=D\Delta v-uv,\\[1mm]
  \displaystyle \rho( {\bf w}_{t}+{\bf w} \cdot \nabla {\bf w})=\mu \Delta {\bf w}+ \nabla p- V_b g u (\rho_b-\rho) {\bf z},\\[1mm]
  \nabla \cdot {\bf w}=0,
\end{cases}
\end{gather}
where $u$ and $v$ denote the bacterial and oxygen concentrations, respectively, and ${\bf w}$ is the fluid velocity governed by the incompressible Navier-Stokes equations with the pure fluid density $\rho$ and viscosity $\mu$. $p$ is a pressure function, $V_b g u (\rho_b-\rho) {\bf z}$ denotes the buoyant force along the upward unit vector ${\bf z}$ where $V_b$ and $\rho_b$ are the bacterial volume and density, respectively, and $g$ is the gravitational constant. With boundary conditions in  \eqref{bd-conditions} and the non-slip boundary condition for the fluid: ${\bf w}|_{\partial \Omega}=0$, the works \cite{chertock2012sinking,lee2015numerical, tuval2005bacterial} have shown that the system \eqref{fluid-chemo} can numerically reproduce the key features of experiment findings in \cite{tuval2005bacterial} in two and three dimensions.

Compared to a large number of results available to the endogenous chemotaxis models with Neumann boundary conditions (cf. \cite{horstmann20031970, Hillen-Painter-JMB, bellomo2015toward}), the basic questions like the global well-posedness of the exogenous chemotaxis system \eqref{multi-D-model} with physical boundary conditions in \eqref{bd-conditions} still remain poorly understood and only very limited analytical results are available so far. The primary obstacle is that the estimate of $\nabla v$, which is needed for the global boundedness of solutions due to the cross-diffusion structure in the first equation of \eqref{multi-D-model}, can not be achieved through the second equation of \eqref{multi-D-model} with the Dirichlet boundary condition which gives no information on $\nabla v$. On the half line $\mathbb{R}_{+}=(0,\infty)$, the existence and stability of the unique stationary solution $(\bar{u}, \bar{v})$ of \eqref{multi-D-model}-\eqref{bd-conditions} with $\phi(v)=\ln v$ was recently established in \cite{carrillo2020boundary} for any $m\geq 0$, where $(\bar{u}, \bar{v})$ is of a boundary (spike, layer) profile as $D>0$ is small. When $\phi(v)=v$, the existence of stationary solutions to \eqref{multi-D-model}-\eqref{bd-conditions} with $m=1$ was proved in \cite{lee2019boundary} for all dimensions and the existence of global weak solutions was established in \cite{winkler2018singular} in one dimension. The local existence of weak solutions to \eqref{fluid-chemo} on the water-drop shaped domain as in \cite{tuval2005bacterial} with \eqref{bd-conditions} and ${\bf w}|_{\partial \Omega}=0$ was proved in \cite{lorz2010coupled}. These appear to be the only results in the literature for the Keller-Segel system \eqref{multi-D-model} subject to the physical boundary conditions given \eqref{bd-conditions}. We also mention another result in \cite{Marcel-Lankeit} where the existence of stationary solutions of \eqref{multi-D-model} with $\phi(v)=v$ and $m=1$ was established for all dimensions when the Dirichlet condition for $v$ in \eqref{bd-conditions} was replaced by a boundary condition based on Henry's law modeling the dissolution of gas in water. The purpose of this paper is to further make a progress in this direction for the Keller-Segel system \eqref{multi-D-model} with linear sensitivity and boundary conditions in \eqref{bd-conditions} on a bounded interval $\mathcal{I}:=(0,1)$. Specifically we consider the following problem
\begin{gather}\label{1-D-model}
 \displaystyle \begin{cases}
 \displaystyle  u _{t}=u _{xx}-(u v _{x})_{x}&\mbox{in}\ \ \mathcal{I},\\[1mm]
\displaystyle v _{t}=D v _{xx} - uv &\mbox{in}\ \ \mathcal{I},\\
(u,v)(x,0)=(u_{0}, v _{0})(x)   &\mbox{in}\ \ \mathcal{I}
 \end{cases}
\end{gather}
subject to the following boundary conditions:
\begin{subequations}\label{bdc}
\begin{numcases}
\displaystyle\makebox[-4pt]{~}  (u _{x}-u v _{x})\vert _{x=0,1}=0, \ \ \ v(0,t)=v(1,t)= v _{\ast}& \mbox{if }$ D>0 $,\label{bd-con-Dnon-0}\\[1mm]
\displaystyle\makebox[-4pt]{~} (u _{x}-u v _{x})\vert _{x=0,1}=0& \mbox{if }$ D=0 $.\label{bdcon-D-0}
\end{numcases}
\end{subequations}
By integrating the first equation of \eqref{1-D-model} along with the boundary condition \eqref{bd-con-Dnon-0}, one immediately finds that cell mass is conserved:
$$\displaystyle\int_{\mathcal{I}} u(x,t)\mathrm{d} x=\int_{\mathcal{I}}u_0(x)\mathrm{d}x:=M,$$
where $M>0$ denotes the initial cell mass. Then the stationary solution $(\bar{u},\bar{v})$ of \eqref{1-D-model} with $D>0$ satisfies
\begin{gather}\label{stat-problem}
\displaystyle \begin{cases}
\displaystyle	\bar{u}_{xx}-(\bar{u}\bar{v}_{x})_{x}=0, &x \in \mathcal{I},\\
	\displaystyle D \bar{v}_{xx}-\bar{u}\bar{v}=0,&x \in \mathcal{I},\\
\displaystyle\int_{\mathcal{I}} \bar{u}\mathrm{d} x=M,\\
		\displaystyle (\bar{u}_{x}-\bar{u}\bar{v}_{x})\vert _{x=0,1}=0,\ \ \bar{v}\vert _{x=0,1}=v _{\ast}.
\end{cases}
\end{gather}
The results of \cite{lee2019boundary} assert that the stationary problem \eqref{stat-problem} with $D>0$ admits a unique non-constant classical solution $ (\bar{u},\bar{v})$ which is of a boundary layer profile as $D>0$ is small. While for the case $ D=0 $, the system \eqref{1-D-model} with \eqref{bdcon-D-0} clearly has a unique constant solution $(M,0)$.

The goal of this paper is to show that if the initial datum $(u_0,v_0)$ is a small perturbation of the stationary solution  $(\bar{u},\bar{v})$, then the system \eqref{1-D-model} with \eqref{bd-con-Dnon-0}-\eqref{bdcon-D-0} admits a unique solution $(u,v)$ satisfying for any $D\geq 0$:
\begin{gather*}
\|(u,v)-(\bar{u},\bar{v})\|_{L^\infty} \to 0 \ \text{exponentially as} \ t \to \infty,
\end{gather*}
where $(\bar{u},\bar{v})=(M,0)$ if $D=0$ and $(\bar{u},\bar{v})$ is the non-constant stationary solution satisfying \eqref{stat-problem} if $D>0$. As we know, this is the first result on the global well-posedness and asymptotic dynamics of the system \eqref{multi-D-model}-\eqref{bd-conditions}. We note that it was shown in \cite{lee2019boundary} that the unique non-constant stationary solution of \eqref{stat-problem} enjoys a boundary layer profile as $D>0$ is small, while only constant stationary solution exits when $D=0$. With this fact and  the boundary conditions in \eqref{bdc}, we may speculate that the solution of \eqref{1-D-model}-\eqref{bd-con-Dnon-0} with $D>0$ will not converge to that of \eqref{1-D-model}-\eqref{bdcon-D-0} with $D=0$ as $D \to 0$, where the boundary layer will arise to correct this discrepancy. Therefore the convergence of solutions to \eqref{1-D-model}-\eqref{bd-con-Dnon-0} with $D>0$ as $D \to 0$ is a very interesting question and we shall investigate it in a separate paper. When $\phi(v)=\ln v$ and the Dirichlet boundary condition for $u$ and Robin boundary condition for $v$ are given, the convergence of solutions for \eqref{multi-D-model} with $\phi(v)=\ln v$ and $m=1$ as $D \to 0$ has been shown in \cite{Hou-Wang-SIMA, Hou-Wang-JMPA}. But they are completely differently from the convergence of \eqref{1-D-model} with \eqref{bd-con-Dnon-0} as $D \to 0$ due to distinct sensitivity function $\phi(v)$ and boundary conditions.
\medskip

{\bf Sketch of proof ideas}. As mentioned previously, the boundary conditions for $v$ in \eqref{bdc} refrain from deriving the estimates of $v_x$ which is, however, necessary to establish the global well-posedness of solutions of \eqref{1-D-model}--\eqref{bdc} due to the cross-diffusion structure in the first equation \eqref{1-D-model}. To overcome this barrier, by observing that the first equation of \eqref{1-D-model} is conserved with zero-flux boundary condition on $u$, we develop an idea by considering the primitive function of $u$ in space, say $\varrho$, and establish the equation of $\varrho$ which no longer has cross-diffusion structure and the Dirichlet boundary condition of $v$ can make essential contributions. As such, we derive the boundedness and stability of $(\varrho, v)$ by the delicate (weighted) energy estimates first and then transfer the results to $(u,v)$. This is our rough idea, and precise procedures are presented in section 3 for the case $D>0$ and in section 4 for $D=0$.
Indeed, the analysis for the case $D=0$ appears to be easier than $D>0$ since its background profile is constant, and thus no weighted estimates are needed. However, since $v$-equation is an ODE and lacks the diffusive dissipation, we need to make full use of the ODE structure along with the explicit formula of $v$ to derive some delicate higher-order estimates which, in turn, requires stronger smallness constraints upon the initial datum compared to the case $ D>0$.

\medskip

The rest of this paper is organized as follows: In Sec. \ref{sec:main_results}, we  state our main results. In Sec. \ref{Asymptiotic-stability}, we investigate the asymptotic behavior of solutions to the system with $ D>0$. Finally, the asymptotic behavior of the solution for the case $ D=0 $ will be proved in Sec. \ref{sec-asymptotic-D0}.

\vspace*{5mm}

\section{Statement of main results} 
\label{sec:main_results}

In this section, we introduce the results on the stationary problem \eqref{stat-problem} from \cite{lee2019boundary} and state our main results on the asymptotic stability of stationary solutions. Throughout the paper, we denote by $ L ^{\infty} $, $ L ^{2} $, $ H _{0}^{1} $ and $ H ^{k} $ the standard function spaces $ L ^{\infty}(\mathcal{I}) $, $ L ^{2}(\mathcal{I}) $, $ H _{0}^{1}(\mathcal{I}) $ and $ H ^{k}(\mathcal{I}) $, respectively. We denote by $\bar{\mathcal{I}}$ the closure of $\mathcal{I} $ and by $ C $ a generic time-independent constant which may take different values in different places. In the sequel, we often omit $\mathcal{I}$ without ambiguity.

\begin{proposition}[Theorem 2.1 in \cite{lee2019boundary}]\label{prop-stat}
For any $ M \in (0,\infty) $, the problem \eqref{stat-problem} with $ D>0 $ admits a unique classical non-constant solution $ (\bar{u}, \bar{v}) \in C ^{1}(\bar{\mathcal{I}}) \cap C ^{\infty}(\mathcal{I}) $ such that
\begin{gather}
\displaystyle   \bar{u}=\frac{M}{\int _{\mathcal{I}}		{\mathop{\mathrm{e}}}^{\bar{v}}\mathrm{d}x}		{\mathop{\mathrm{e}}}^{\bar{v}}, \ \
 \bar{u}>0,\ \ \  0< \bar{v} \leq v _{\ast}\ \ \ \mbox{for any}\ \  x \in \bar{\mathcal{I}}. \label{stat-iden}
\end{gather}

 \end{proposition}
 ~\\

Our first result is the asymptotic stability of stationary solutions obtained in Proposition \ref{prop-stat} for the initial-boundary value problem \eqref{1-D-model}, \eqref{bd-con-Dnon-0} as time goes to infinity.\\
\begin{theorem}\label{thm-asymtoptic}
Suppose that $ u _{0}\in H ^{1} $ and $ v _{0}\in H ^{2} $ with $ u _{0}\geq 0 $, $ v _{0}\geq 0 $ such that $ \int _{\mathcal{I}}u _{0}\mathrm{d}x=M $. Let $ (\bar{u}, \bar{v}) $ be the stationary solution given in Proposition \ref{prop-stat} with $ \int _{\mathcal{I}}\bar{u}\mathrm{d}x=M$ and define
\begin{gather*}
\displaystyle \varphi _{0}(x)=\int _{0}^{x}\left( u _{0}(y)-\bar{u}(y) \right) \mathrm{d}y.
\end{gather*}
Then there exists a constant $ \delta _{0}>0 $ such that if
\begin{gather*}
\displaystyle  \|\varphi _{0}\|_{H ^{1}}^{2}+\|v _{0}- \bar{v}\|_{L ^{2}}^{2}\leq \delta _{0},
\end{gather*}
then the initial-boundary value problem \eqref{1-D-model}, \eqref{bd-con-Dnon-0} admits a unique global solution $ (u,v) $ satisfying
\begin{gather*}
\displaystyle  u \in C([0,\infty); H ^{1})\cap L ^{2}(0,\infty;H ^{2}),\ \ v \in C([0,\infty);H ^{2})\cap L ^{2}(0,\infty;H ^{3}),
\end{gather*}
and the following asymptotic decay:
\begin{gather}\label{decay-in-thm}
\displaystyle \|(u- \bar{u}, v- \bar{v})(\cdot,t)\|_{L ^{\infty}} \leq C 		{\mathop{\mathrm{e}}}^{-\alpha t}  \ \ \mbox{for any }t \geq 0,
\end{gather}
where $ \alpha $ and $ C $ are positive constants independent of $ t $.
\end{theorem}
~\\
When $ D=0 $, \eqref{1-D-model} becomes a PDE-ODE system which has a unique constant steady state $ (M,0) $ with $ M= \int _{\mathcal{I}}u _{0}\mathrm{d}x $ satisfying the boundary condition \eqref{bdcon-D-0}. Then we have our second result below.
\begin{theorem}\label{thm-stabiliy-D0}
Let $ (u _{0}, v _{0})\in H ^{1} \times H ^{2} $ with $ u _{0} \geq 0 $, $ v _{0} \geq 0 $ such that $ \int _{\mathcal{I} }u _{0} \mathrm{d}x =M$ and define
\begin{gather*}
\displaystyle w _{0}(x)= \int _{0}^{x}(u _{0}(y)-M) \mathrm{d}y.
\end{gather*}
Then there exists a constant $ \delta _{1}>0 $ such that if
\begin{gather*}
\displaystyle  \|w _{0}\|_{H ^{1}}^{2}+\|v _{0}\|_{H ^{1}}^{2}\leq \delta _{1},
\end{gather*}
then the initial-boundary value problem \eqref{1-D-model}, \eqref{bdcon-D-0} admits a unique solution $ (u,v) $ in $ \mathcal{I}\times (0,\infty) $ satisfying
\begin{gather*}
\displaystyle  u \in C([0,\infty); H ^{1})\cap L ^{2}(0,\infty;H ^{2}),\ \ v \in C([0,\infty);H ^{2}).
\end{gather*}
Furthermore, we have the following decay estimates:
\begin{gather}\label{decay-in-thm-D0}
\displaystyle \|(u- M, v)(\cdot,t)\|_{L ^{\infty}} \leq C 		{\mathop{\mathrm{e}}}^{-\alpha _{0} t}  \ \ \mbox{for any } t >0,
\end{gather}
where $ \alpha _{0} $ and $ C>0 $ are positive constants independent of $ t $.
\end{theorem}

\vspace{5mm}
\section{Asymptotic stability for the case \texorpdfstring{$ D>0 $}{D>0}} 
\label{Asymptiotic-stability}
In this section, we will study the asymptotic stability of the steady state of \eqref{1-D-model}, \eqref{bdcon-D-0} for $ D>0 $ by the method of energy estimates. Before proceeding, we present an well-known inequality that will be frequently used in the sequel. 
\begin{lemma}[cf. \cite{Nirenberg-1966-inequaltiy}]
For any $ f \in H ^{1}(\mathcal{I})$, there exists a constant $ c _{1}>0 $ such that
\begin{gather}\label{sobole}
\displaystyle \|f\|_{L ^{\infty}} \leq c_{1}\Big( \|f\|_{L ^{2}}^{\frac{1}{2}}\|f _{x}\|_{L ^{2}}^{\frac{1}{2}}+\|f\|_{L ^{2}} \Big) .
\end{gather}
Furthermore, if $ f \in H _{0}^{1}(\mathcal{I}) $, then it holds that
\begin{gather}\label{sobolve-l-inty-zero}
\displaystyle \|f\|_{L ^{\infty}} \leq c_{2} \|f\|_{L ^{2}}^{\frac{1}{2}}\|f _{x}\|_{L ^{2}}^{\frac{1}{2}} \ \ \ \mbox{and}\ \ \ \|f\|_{L ^{\infty}}\leq c _{3}\|f _{x}\|_{L ^{2}}
\end{gather}
for some constants $ c _{2}, c _{3}>0 $.
\end{lemma}

\subsection{\emph{A priori} estimates} 
\label{sub:em}
First of all, integrating the first equation in \eqref{stat-problem}, we see that the stationary solution $ (\bar{u},\bar{v}) $ satisfies
\begin{gather}\label{stat-refor}
\displaystyle \begin{cases}
\displaystyle	\bar{u} _{x}-\bar{u}\bar{v}_{x}=0,\\
	\displaystyle D \bar{v}_{xx}-\bar{u}\bar{v}=0,\\
	\displaystyle \bar{v}(0)=\bar{v}(1)=v _{\ast},
\end{cases}
\end{gather}
with $ \int _{\mathcal{I}} \bar{u} \mathrm{d}x=M $. In view of the zero-flux boundary condition in \eqref{bd-con-Dnon-0} for $ u$, we know that the mass of the bacteria is conserved for all time. This along with the fact $ \int _{\mathcal{I}}u _{0}\mathrm{d}x=\int _{\mathcal{I}}\bar{u}\mathrm{d}x =M$ implies that
\begin{gather*}
\displaystyle \int _{\mathcal{I}}(u(x,t)-\bar{u}(x))\mathrm{d}x=0
\end{gather*}
for any $ t \geq 0 $. Define
\begin{gather*}
\displaystyle \varphi(x,t)= \int _{0}^{x}(u(y,t)-\bar{u}(y))\mathrm{d}y,\ \ \psi=v- \bar{v},
\end{gather*}
that is
\begin{gather}\label{u-v-vfipsi}
\displaystyle  u= \varphi _{x}+\bar{u},\ \ \ v=\psi+ \bar{v}.
\end{gather}
Substituting \eqref{u-v-vfipsi} into \eqref{1-D-model}, integrating the first equation with respect to $ x $ and using \eqref{stat-refor}, we obtain the following perturbation equations:
\begin{gather}\label{vfi-psi-eq}
\displaystyle \begin{cases}
	\displaystyle \varphi _{t}= \varphi _{xx}-  \varphi _{x}\bar{v}_{x}-\bar{u}\psi _{x}- \varphi _{x}\psi _{x},\\[1mm]
  \displaystyle \psi _{t}=D \psi _{xx}-\bar{u}\psi- \bar{v}\varphi _{x} -  \varphi _{x} \psi,
\end{cases}
\end{gather}
with the initial datum
\begin{gather}
\displaystyle  (\varphi,\psi)(x,0)=(\varphi _{0},\psi _{0})=\left( \int _{0}^{x}(u _{0}(y)-\bar{u}(y))\mathrm{d}y,\, v _{0}- \bar{v}  \right)
\end{gather}
and the boundary conditions
\begin{gather}\label{bd-con-refor}
 \displaystyle(\varphi,\psi)(0,t)=(\varphi,\psi)(1,t)=0.
 \end{gather}
 By the standard fixed point theorems (cf. \cite{nishida-1978,zhaokun-2015-JDE}), one can prove the local existence of solutions to the initial-boundary value problem \eqref{vfi-psi-eq}--\eqref{bd-con-refor}. Precisely, for any $ T>0 $, if we define
\begin{align*}
\displaystyle X(0,T):=\{ (\varphi,\psi)\vert\, &\varphi \in C([0,T];H _{0}^{1}\cap H ^{2}) \cap L ^{2}(0,T;H ^{3}),\\
 & \psi \in C([0,T]; H _{0}^{1}\cap H ^{2})\cap\in L ^{2}(0,T;H ^{3})\}
\end{align*}
and denote
\begin{gather*}
 \displaystyle N(T):=\sup _{0 \leq t \leq T}\left( \|\varphi\|_{H ^{2}}^{2} +\|\psi\|_{H ^{2}}^{2}\right),
 \end{gather*}
 then we have the following local existence result.
\begin{proposition}[Local existence]\label{prop-local}
Let $\varphi _{0}\in H _{0}^{1}\cap H ^{2} $ and $ \psi _{0} \in H _{0}^{1}\cap H ^{2} $ such that
\begin{gather*}
\displaystyle \varphi _{0x}+\bar{u} \geq 0,\ \ \psi _{0}+\bar{v}\geq 0
\end{gather*}
for any $ x \in \mathcal{I} $. Then there exists a positive constant $ T _{0} $ depending on $ N(0) $, $ \bar{u} $ and $ \bar{v} $ such that the initial-boundary value problem \eqref{vfi-psi-eq}--\eqref{bd-con-refor} admits a unique solution $ (\varphi,\psi)\in X(0,T _{0}) $ satisfying $ N(T _{0}) \leq 2 N(0) $ and
\begin{gather*}
\displaystyle \varphi _{x}+\bar{u} \geq 0,\ \ \ \psi +\bar{v} \geq 0
\end{gather*}
for any $ (x,t)\in \mathcal{I}\times[0,T _{0}) $.
\end{proposition}

In order to study the asymptotic behavior of solutions to the problem \eqref{1-D-model}, \eqref{bd-con-Dnon-0}, we first establish the global existence result for the initial-boundary value problem \eqref{vfi-psi-eq}--\eqref{bd-con-refor}.
\begin{proposition}\label{prop-global-refor}
Assume the conditions of Proposition \ref{prop-local} hold. Then there exists a positive constant $ \delta _{1} $, such that if $ \|\varphi _{0}\|_{H ^{1}}^{2}+\|\psi _{0}\|_{L ^{2}}^{2}\leq \delta _{1} $, then the problem \eqref{vfi-psi-eq}--\eqref{bd-con-refor} admits a unique global solution $ (\varphi, \psi)\in X(0,\infty) $ which satisfies that for all $ t \geq 0 $,
\begin{gather}\label{pro-global-eti1}
\displaystyle  \|\varphi(\cdot,t)\|_{H ^{1}}^{2}+\|\psi(\cdot,t)\|_{L ^{2}}^{2} \leq C 		{\mathop{\mathrm{e}}}^{- \alpha _{1} t},\ \ \|\varphi _{xx}(\cdot,t)\|_{L ^{2}}^{2} + \|\psi _{x}(\cdot,t)\|_{H ^{1}}^{2} \leq C
\end{gather}
for some constants $ \alpha _{1}>0 $ and $ C>0 $ independent of $ t $.
\end{proposition}

To ensure the global existence of solutions to the problem \eqref{vfi-psi-eq}--\eqref{bd-con-refor}, by the local existence result and the standard continuation argument, it suffices to prove the following \emph{a priori} estimates.
\begin{proposition}[\emph{A priori} estimates]\label{prop-aprioriesti}
For any $ T>0 $ and any solution  $ (\varphi,\psi)\in X(0,T) $ to the problem \eqref{vfi-psi-eq}--\eqref{bd-con-refor} with $ (\varphi _{0},\psi _{0})\in H ^{2} $, there exists a suitably small $ C _{0}>0 $ independent of $ T $ such that if $\|\varphi _{0}\|_{H ^{1}}^{2}+\|\psi _{0}\|_{L ^{2}}^{2}\leq C _{0}  $, then we have
\begin{align*}
\displaystyle  \|\varphi(\cdot,t)\|_{H ^{1}}^{2}+\|\psi(\cdot,t)\|_{L ^{2}}^{2} \leq  C		{\mathop{\mathrm{e}}}^{- \alpha _{1} t},\ \ \|\varphi _{xx}(\cdot,t)\|_{L ^{2}}^{2} + \|\psi _{x}(\cdot,t)\|_{H ^{1}}^{2} \leq C\ \ \mbox{in}\ \ [0,T]
\end{align*}
and
\begin{gather*}
\displaystyle \int _{0}^{t}\left( \|\varphi \|_{H ^{3}} ^{2}+\|\psi \|_{ H ^{3}}^{2}+\|\varphi _{\tau} \|_{H ^{1}} ^{2}+\|\psi _{\tau} \|_{ H ^{1}}^{2}  \right)  \mathrm{d}\tau \leq C \left( \|\varphi _{0}\|_{H ^{2}} ^{2}+\|\psi _{0}\|_{ H ^{2}}^{2} \right)  \ \ \mbox{in}\ \ [0,T],
\end{gather*}
where  $ \alpha _{1} $ and $ C $ are positive constants independent of $ T $.
\end{proposition}

We shall prove Proposition \ref{prop-aprioriesti} by the argument of \emph{a priori }assumption. That is, we first assume that the solution $ (\varphi,\psi) $ to the problem \eqref{vfi-psi-eq}--\eqref{bd-con-refor} satisfy for any $ t \in [0,T] $,
\begin{align}\label{aprori-assum}
\displaystyle  \|\varphi(\cdot,t)\|_{H ^{1}}+\|\psi(\cdot,t)\|_{L ^{2}} \leq 2\delta,\ \ \|\varphi _{xx}(\cdot,t)\|_{L ^{2}} + \|\psi _{x}(\cdot,t)\|_{H ^{1}} \leq 2 \sigma\ \ \mbox{in}\ \ [0,T],
\end{align}
where $ 0<\delta<1 $ and $ M _{1} $ are positive constants to be determined later, and then derive the \emph{a priori} estimates with \eqref{apriori-ass-D-0} to ensure the global existence of solutions. Finally, we show that the solution exactly satisfies the \emph{a priori} assumption \eqref{apriori-ass-D-0} and close the argument. Before proceeding, we note that by \eqref{aprori-assum} along with \eqref{sobole}, \eqref{sobolve-l-inty-zero} and \eqref{bd-con-refor}, we get
\begin{gather}\label{small-implica}
\displaystyle \left\|\varphi\right\|_{L ^{\infty}}\leq C \delta , \ \ \|\psi\|_{L ^{\infty}}\leq C\delta ^{\frac{1}{2}}\sigma ^{\frac{1}{2}},  \ \ \|\varphi _{x}\|_{L ^{\infty}}\leq C \delta +C \delta ^{\frac{1}{2}}\sigma ^{\frac{1}{2}}
\end{gather}
for some constant $ C>0 $ independent of $ \delta $, $ \sigma$ and $ T $.

\vspace{2mm}

The following simple properties on the stationary solution are of importance in studying the asymptotic behavior of solutions.
\begin{lemma}\label{lem-stat}
Let $ (\bar{u},\bar{v}) $ be the stationary solution of \eqref{1-D-model}, \eqref{bd-con-Dnon-0} stated in Proposition \ref{prop-stat}. Then it holds that
\begin{gather}\label{bd-stat}
\displaystyle 0<C _{1}^{-1} \leq \bar{u},\, \bar{v} \leq C _{1}
\end{gather}
for some constant $ C _{1}>0 $, and that
\begin{gather}\label{stat-vx-bd}
\displaystyle D\bar{v}_{x}^{2} \leq\bar{v}^{2}\bar{u}.
\end{gather}

 \end{lemma}
 \begin{proof}
 According to Proposition \ref{prop-stat}, the proof of \eqref{bd-stat} is trivial and hence we prove \eqref{stat-vx-bd} only. Since $ 0 <\bar{v}(x) \leq v _{\ast} $ for any $ x \in \bar{\mathcal{I}} $, then there exists an $ x _{0} \in (0,1) $ such that
 \begin{gather*}
  \displaystyle 0<\bar{v}(x _{0})= \min _{x \in \bar{\mathcal{I}}}\bar{v}(x)\ \ \mbox{and}\ \ \bar{v}_{x}( x_{0})=0.
  \end{gather*}
  Multiplying the second equation in \eqref{stat-problem} by $ \bar{v} _{x} $ followed by an integration from $ x _{0} $ to $ x $, we have
\begin{gather*}
\displaystyle  \frac{D}{2}\bar{v} _{x}^{2}=\int _{x _{0}}^{x} \bar{u}\bar{v}v _{y}\mathrm{d}y =  \lambda\int _{\bar{v}(x _{0})}^{\bar{v}(x)} s     {\mathop{\mathrm{e}}}^{s}\mathrm{d}s \leq \lambda \int _{0}^{\bar{v}} s 		{\mathop{\mathrm{e}}}^{s}\mathrm{d}s,
\end{gather*}
with $ \lambda= \frac{M}{\int _{\mathcal{I}}		{\mathop{\mathrm{e}}}^{\bar{v}}\mathrm{d}x} $,
where we have used the following identity
\begin{gather}\label{baru}
\displaystyle \bar{u}= \frac{M}{\int _{\mathcal{I}}		{\mathop{\mathrm{e}}}^{\bar{v}}\mathrm{d}x}		{\mathop{\mathrm{e}}}^{\bar{v}}=:\lambda{\mathop{\mathrm{e}}}^{\bar{v}}
\end{gather}
from \eqref{stat-iden}. Hence, we get, thanks to \eqref{baru} and integration by parts, that
\begin{align*}
\displaystyle \frac{D \bar{v}_{x}^{2}}{2}& \leq \lambda \frac{\bar{v}^{2}		{\mathop{\mathrm{e}}}^{\bar{v}}}{2}- \frac{\lambda}{2}\int _{0}^{\bar{v}}s ^{2}		{\mathop{\mathrm{e}}}^{s}\mathrm{d}s  =\frac{\bar{v}^{2}\bar{u}}{2}- \frac{\lambda}{2}\int _{0}^{\bar{v}}s ^{2}		{\mathop{\mathrm{e}}}^{s}\mathrm{d}s \leq \frac{\bar{v}^{2}\bar{u}}{2}.
\end{align*}
The proof is completed.
 \end{proof}
 Now let us turn to estimates on the solution $ (\varphi, \psi) $. We begin with the following weighted $ L ^{2} $ estimate.
 \begin{lemma}\label{lem-l2-esti}
 For any solution $ (\varphi,\psi)\in X(0,T) $ to the problem \eqref{vfi-psi-eq}--\eqref{bd-con-refor} satisfying \eqref{aprori-assum}, it holds that
\begin{gather}\label{con-1-l2}
\displaystyle \int _{\mathcal{I}}\left( \frac{\varphi ^{2}}{\bar{u}}+\frac{\psi^{2}}{\bar{v}} \right)\mathrm{d}x +\int _{0}^{t}\left( \|\varphi _{x}\|_{L ^{2}}^{2}+\|\psi _{x}\|_{L ^{2}}^{2} \right) \mathrm{d}\tau \leq C  \left( \|\varphi _{0}\|_{L ^{2}}^{2}+\|\psi _{0}\|_{L ^{2}}^{2} \right)
\end{gather}
for any $ t \in [0,T] $, provided that $ \delta $ is suitably small, where $ C >0 $ is a constant independent of $ t $.
 \end{lemma}
 \begin{proof}
   Multiplying the first equation in \eqref{vfi-psi-eq} by $ \frac{\varphi}{\bar{u}} $ followed by an integration over $ \mathcal{I} $, we get after using integration by parts that
\begin{align}\label{diff-iden-firs}
&\displaystyle  \frac{1}{2}\frac{\mathrm{d}}{\mathrm{d}t}\int _{\mathcal{I} }\frac{ \varphi ^{2}}{\bar{u}} \mathrm{d}x+\int _{\mathcal{I}}\frac{ \varphi _{x}^{2}}{\bar{u}} \mathrm{d}x
 =-\int _{\mathcal{I} } \varphi  \varphi _{x}\left [\bar{v}_{x}+\left( \frac{1}{\bar{u}} \right)_{x}  \right]  \mathrm{d}x-\int _{\mathcal{I}} \psi _{x} \varphi \mathrm{d}x-\int _{\mathcal{I} }\frac{ \psi _{x} \varphi  \varphi _{x}}{\bar{u}} \mathrm{d}x.
 \end{align}
 By the first equation in \eqref{stat-refor}, we have
 \begin{align*}
 \displaystyle   \left( \frac{1}{\bar{u}} \right)_{x}+ \frac{\bar{v}_{x}}{\bar{u}}=- \frac{1}{\bar{u}^{2}}\left( \bar{u}_{x}-\bar{u}\bar{v}_{x} \right)=0,
 \end{align*}
 and thus
 \begin{gather}\label{concelat}
 \displaystyle  -\int _{\mathcal{I} } \varphi  \varphi _{x}\left [\bar{v}_{x}+\left( \frac{1}{\bar{u}} \right)_{x}  \right]  \mathrm{d}x =0.
 \end{gather}
 Thanks to \eqref{small-implica} and the Cauchy-Schwarz inequality, we get
 \begin{gather}\label{nonlie-vfi}
 \displaystyle  \int _{\mathcal{I} }\frac{ \psi _{x} \varphi  \varphi _{x}}{\bar{u}} \mathrm{d}x \leq \|\varphi\|_{L ^{\infty}} \|\psi _{x}\|_{L ^{2}}\|\varphi _{x}\|_{L ^{2}} \leq C \delta \left( \|\varphi _{x}\|_{L ^{2}}^{2}+\|\psi _{x}\|_{L ^{2}}^{2} \right).
 \end{gather}
 Inserting \eqref{concelat} and \eqref{nonlie-vfi} into \eqref{diff-iden-firs} gives
 \begin{gather}\label{diff-con-phi}
 \displaystyle  \frac{1}{2}\frac{\mathrm{d}}{\mathrm{d}t}\int _{\mathcal{I} }\frac{\varphi ^{2}}{\bar{u}} \mathrm{d}x+ \int _{\mathcal{I}}\frac{ \varphi _{x}^{2}}{\bar{u}} \mathrm{d}x   \leq -\int _{\mathcal{I} } \psi _{x} \varphi \mathrm{d}x+C \delta \left( \|\varphi _{x}\|_{L ^{2}}^{2}+\|\psi _{x}\|_{L ^{2}}^{2} \right).
 \end{gather}
To proceed, multiplying the second equation in \eqref{vfi-psi-eq} by $ \frac{\psi}{\bar{v}}$ and then integrating the resulting equation over $ \mathcal{I} $, we have
 \begin{align}\label{psi-esti-basic-bd}
 \displaystyle\frac{1}{2}\frac{\mathrm{d}}{\mathrm{d}t} \int _{\mathcal{I} }\frac{ \psi ^{2}}{\bar{v}} \mathrm{d}x + D \int _{\mathcal{I}}\frac{ \psi _{x}^{2}}{\bar{v}} \mathrm{d}x
    &\displaystyle =-  \int _{\mathcal{I} } \left[ \frac{\bar{u}}{\bar{v}} \psi ^{2} +D \psi  \psi _{x}\Big( \frac{1}{\bar{v}} \Big)_{x} \right] \mathrm{d}x -  \int _{\mathcal{I} } \psi  \varphi _{x}  \mathrm{d}x - \int _{\mathcal{I} }\frac{ \psi ^{2} \varphi _{x}}{\bar{v}} \mathrm{d}x,
        \end{align}
  where, by virtue of \eqref{stat-problem} and \eqref{stat-vx-bd}, it holds that
  \begin{align*}
  \displaystyle -  \int _{\mathcal{I} } \left[ \frac{\bar{u}}{\bar{v}} \psi ^{2} +D \psi  \psi _{x}\Big( \frac{1}{\bar{v}} \Big)_{x} \right] \mathrm{d}x & = - \int _{ \mathcal{I}}  \psi ^{2}\left( \frac{\bar{u}}{\bar{v}}+ \frac{D \bar{v}_{xx}}{2\bar{v}^{2}}- \frac{D\bar{v}_{x}^{2}}{\bar{v}^{3}} \right) \mathrm{d}x
   =- \int _{ \mathcal{I}}  \psi ^{2}\left(\frac{3}{2}\frac{\bar{u}}{\bar{v}}- D\frac{\bar{v}_{x}^{2}}{\bar{v}^{3}} \right) \mathrm{d}x
    \nonumber \\
    &\displaystyle \leq - \frac{3}{2}\int _{\mathcal{I}}\frac{\bar{u}}{\bar{v}}\psi ^{2}\mathrm{d}x + \int _{\mathcal{I}}\frac{\bar{u}}{\bar{v}}\psi ^{2}\mathrm{d}x \leq - \frac{1}{2}\int _{ \mathcal{I}} \frac{\bar{u}}{\bar{v}} \psi ^{2} \mathrm{d }x.
  \end{align*}
  For the last term on the right hand side of \eqref{psi-esti-basic-bd}, by \eqref{sobolve-l-inty-zero}, \eqref{aprori-assum}, \eqref{bd-stat} and the H\"older inequality, we get
  \begin{align*}
  \displaystyle  - \int _{\mathcal{I} }\frac{ \psi ^{2} \varphi _{x}}{\bar{v}}\mathrm{d}x
  & \leq C \|\psi \|_{L ^{\infty}}\|\psi  \|_{L ^{2}}\|\varphi _{x}\|_{L ^{2}} \leq C \|\varphi _{x}\|_{L ^{2}}\|\psi _{x}\|_{L ^{2}}^{2}
  \leq C \delta \|\psi _{x}\|_{L ^{2}}^{2}.
  \end{align*}
  We thus have from \eqref{psi-esti-basic-bd} that
  \begin{gather}\label{diff-con-psi}
  \displaystyle  \displaystyle \frac{1}{2}\frac{\mathrm{d}}{\mathrm{d}t} \int _{\mathcal{I} }\frac{ \psi ^{2}}{\bar{v}} \mathrm{d}x + D \int _{\mathcal{I}}\frac{ \psi _{x}^{2}}{\bar{v}} \mathrm{d}x +\frac{1}{2}\int _{ \mathcal{I}} \frac{\bar{u}}{\bar{v}} \psi ^{2} \mathrm{d }x  \leq -  \int _{\mathcal{I} } \psi  \varphi _{x}  \mathrm{d}x + C \delta  \|\psi _{x}\| _{L ^{2}}^{2}.
  \end{gather}
  Adding \eqref{diff-con-psi} with \eqref{diff-con-phi}, we then arrive at
  \begin{align*}
  \displaystyle  \displaystyle \frac{1}{2}\frac{\mathrm{d}}{\mathrm{d}t} \int _{\mathcal{I} }\left( \frac{ \varphi ^{2}}{\bar{u}}+\frac{ \psi ^{2}}{\bar{v}} \right)  \mathrm{d}x +\int _{\mathcal{I}}\left(\frac{\varphi _{x}^{2}}{\bar{u}}+ D \frac{\psi _{x}^{2}}{\bar{v}} + \frac{\bar{u}\psi ^{2}}{2 \bar{v}}\right) \mathrm{d}x
  \leq C \delta \left( \|\varphi _{x}\|_{L ^{2}}^{2}+\|\psi _{x}\|_{L ^{2}}^{2} \right).
  \end{align*}
  This along with \eqref{bd-stat} implies that
  \begin{align*}
  \displaystyle \frac{1}{2}\frac{\mathrm{d}}{\mathrm{d}t} \int _{\mathcal{I} }\left( \frac{ \varphi ^{2}}{\bar{u}}+\frac{ \psi ^{2}}{\bar{v}} \right)  \mathrm{d}x +\frac{1}{C _{1}}\min\{1,D\} \left( \|\varphi _{x}\|_{L ^{2}}^{2}+\|\psi _{x}\|_{L ^{2}}^{2} \right) \leq  C \delta \left( \|\varphi _{x}\|_{L ^{2}}^{2}+\|\psi _{x}\|_{L ^{2}}^{2} \right).
  \end{align*}
 Therefore it holds that
  \begin{align}\label{diff-con-1-l2}
  \displaystyle \frac{\mathrm{d}}{\mathrm{d}t} \int _{\mathcal{I} }\left( \frac{ \varphi ^{2}}{\bar{u}}+\frac{ \psi ^{2}}{\bar{v}} \right)  \mathrm{d}x + \beta \left( \|\varphi _{x}\|_{L ^{2}} ^{2}+\|\psi_{x}\|_{L ^{2}}^{2} \right) \leq 0,
  \end{align}
  provided that
  \begin{gather}\label{deta-constrian-lem1}
  \displaystyle C \delta  \leq \frac{1}{2C _{1}} \min\{1,D\}=: \frac{\beta}{2}.
   \end{gather}
   Integrating \eqref{diff-con-1-l2} over $ (0,t) $, we then get \eqref{con-1-l2}. The proof of Lemma \ref{lem-l2-esti} is complete.
 \end{proof}

 In the next lemma, we are going to derive the estimate on $ \varphi _{x} $.
 \begin{lemma}\label{lem-smallness-assum-clos}
 Under the conditions of  Lemma \ref{lem-l2-esti}, let $ (\varphi,\psi)\in X(0,T) $ be a solution to the initial-boundary value problem \eqref{vfi-psi-eq}--\eqref{bd-con-refor} satisfying \eqref{aprori-assum}. Then for any given $ \sigma>0 $, it holds for any $ t \in [0,T]$ that
\begin{gather}\label{con-decay-l2}
\displaystyle  \|\varphi \| _{H ^{1}}^{2}+\|\psi \| _{L ^{2}}^{2} 	\leq  C \left(\|\varphi _{0}\|_{H^{1}} ^{2}+\|\psi _{0}\|_{L ^{2}}^{2}\right)	{\mathop{\mathrm{e}}}^{-\alpha _{1} t},
\end{gather}
provided $ \delta $ is suitably small, where $ \alpha _{1} $ is determined by \eqref{al-defi}.
 \end{lemma}
       \begin{proof}
        Multiplying the first equation in \eqref{vfi-psi-eq} by $\varphi_{t} $ and integrating the resulting equation over $ \mathcal{I} $, we get by integration by parts that
\begin{align}\label{diff-vfi-x-l2}
&\displaystyle \frac{1}{2}\frac{\mathrm{d}}{\mathrm{d}t}\int _{\mathcal{I} }\varphi _{x}^{2} \mathrm{d}x+
\int _{\mathcal{I} } \varphi _{t}^{2} \mathrm{d}x  = -\int _{\mathcal{I} } \varphi _{x} \bar{v} _{x} \varphi _{t} \mathrm{d}x- \int _{\mathcal{I} }\bar{u} \psi _{x}\varphi _{t} \mathrm{d}x-\int _{\mathcal{I} } \varphi _{t} \varphi _{x} \psi _{x} \mathrm{d}x.
\end{align}
Next we estimate the terms on the right hand side of \eqref{diff-vfi-x-l2}. By \eqref{bd-stat}, \eqref{stat-vx-bd} and the Cauchy-Schwarz inequality, we deduce that
\begin{align}\label{vfi-x-err1}
\displaystyle  -\int _{\mathcal{I} } \varphi _{x}\bar{v} _{x} \varphi _{t} \mathrm{d}x &\leq \|\bar{v}_{x}\|_{L ^{\infty}}\| \varphi _{x}\|_{L ^{2}}\| \varphi _{t}\|_{L ^{2}} \leq  \varepsilon \| \varphi _{t}\|_{L ^{2}}^{2}+C _{\varepsilon}\| \varphi _{x}\|_{L ^{2}}^{2}, \\
\displaystyle\displaystyle -\int _{\mathcal{I} }\bar{u}\psi _{x} \varphi _{t} \mathrm{d}x &\leq  \|\bar{u}\| _{L ^{\infty
}} \|\varphi _{t}\|_{L ^{2}}\|\psi _{x}\|_{L ^{2}} \leq\varepsilon \|\varphi _{t}\| _{L ^{2}}^{2}+C _{\varepsilon}\left\|\psi _{x}\right\| _{L ^{2}} ^{2}
 \nonumber
\end{align}
for any $ \varepsilon>0 $. For the last term on the right hand side of \eqref{diff-vfi-x-l2}, it follows from \eqref{small-implica} and the Cauchy-Schwarz inequality that
\begin{align}\label{vfi-x-err3}
\displaystyle -\int _{\mathcal{I} } \varphi _{t} \varphi _{x} \psi _{x} \mathrm{d}x \leq \| \varphi _{x}\|_{L ^{\infty}}\| \varphi _{t}\|_{L ^{2}} \left\|\psi _{x}\right\| _{L ^{2}} \leq C \left( \delta+C \delta ^{\frac{1}{2}}\sigma ^{\frac{1}{2}} \right)  \left( \| \varphi _{t}\| _{L ^{2}}^{2}+ \|\psi _{x}\|  _{L ^{2}}^{2}\right).
\end{align}
Substituting \eqref{vfi-x-err1}--\eqref{vfi-x-err3} into \eqref{diff-vfi-x-l2} and choosing $ \delta $ small enough such that
\begin{gather}\label{DE-M1-const}
\displaystyle C \left( \delta +C \delta ^{\frac{1}{2}}\sigma ^{\frac{1}{2}} \right) \leq \frac{1}{4},
\end{gather}
we get after taking $ \varepsilon $ suitably small that
\begin{align}\label{vfi-x}
\displaystyle   \frac{1}{2}\frac{\mathrm{d}}{\mathrm{d}t}\int _{\mathcal{I} }\varphi _{x}^{2} \mathrm{d}x+ \frac{1}{2}\int _{\mathcal{I} } \varphi _{t}^{2} \mathrm{d}x \leq C \|\psi _{x}\|_{L ^{2}}^{2}.
\end{align}
Adding \eqref{vfi-x} with \eqref{diff-con-1-l2} multiplied by a sufficiently large constant $ K >0$ such that $ \frac{1}{2} K \beta>C $, it then follows that
\begin{align}\label{vfi-x-diff}
\displaystyle \frac{\mathrm{d}}{\mathrm{d}t} \int _{\mathcal{I}}\left( \frac{ \varphi ^{2}}{\bar{u}}+\frac{ \psi ^{2}}{\bar{v}}+\varphi _{x}^{2} \right)\mathrm{d}x+ \beta _{1}\left( \|\varphi _{x}\|_{L ^{2}}^{2}+\|\psi _{x}\| _{L ^{2}}^{2} +\|\varphi _{t}\|_{L ^{2}}^{2}\right) \leq 0
\end{align}
for some $ \beta _{1}>0 $. Multiplying \eqref{vfi-x-diff} by $ {\mathop{\mathrm{e}}}^{\alpha _{1} t} $ with $ \alpha _{1} $ being a constant to be determined later, we have
   \begin{align}\label{decay-exp-l2}
   &\displaystyle  \frac{\mathrm{d}}{\mathrm{d}t} \left\{ 		{\mathop{\mathrm{e}}}^{\alpha _{1} t} \int _{\mathcal{I} }\Big( \frac{ \varphi ^{2}}{\bar{u}}+\frac{ \psi ^{2}}{\bar{v}} +\varphi _{x}^{2} \Big)  \mathrm{d}x  \right\}+
    		{\mathop{\mathrm{e}}}^{\alpha _{1} t} \left[ \beta _{1}( \|\varphi _{x}\|_{L ^{2}} ^{2}+\|\psi _{x}\|_{L ^{2}}^{2} )-  \int _{\mathcal{I} }\alpha _{1}\Big( \frac{ \varphi ^{2}}{\bar{u}}+\frac{ \psi ^{2}}{\bar{v}} +\varphi _{x}^{2}\Big)  \mathrm{d}x \right]   \leq 0,
   \end{align}
   where we have ignored $ \|\varphi _{t}\|_{L ^{2}}^{2} $ on the left hand side of \eqref{vfi-x-diff} due to $ \beta _{1}>0 $. By \eqref{sobolve-l-inty-zero}, \eqref{bd-con-refor}, \eqref{bd-stat} and the Sobolev inequality $ \|f\|_{L ^{2}}\leq C \|f _{x}\|_{L ^{2}} $ for any $ f \in H _{0}^{1} $ asserting
   \begin{gather*}
   \displaystyle   \int _{\mathcal{I} }\left( \frac{ \varphi ^{2}}{\bar{u}}+\frac{ \psi ^{2}}{\bar{v}} \right)  \mathrm{d}x \leq C\left( \|\varphi _{x}\|_{L ^{2}} ^{2}+\| \psi _{x}\|_{L ^{2}}^{2} \right)
   \end{gather*}
for some constant $ C>0 $, we get from \eqref{decay-exp-l2} that
\begin{align}\label{l2-full-diff}
\displaystyle   \frac{\mathrm{d}}{\mathrm{d}t} \left\{ 		{\mathop{\mathrm{e}}}^{\alpha _{1} t} \int _{\mathcal{I} }\left( \frac{ \varphi ^{2}}{\bar{u}}+\frac{ \psi ^{2}}{\bar{v}} +\varphi _{x}^{2}\right)  \mathrm{d}x  \right\} \leq 0,
\end{align}
provided that
\begin{gather}\label{al-defi}
\displaystyle  \alpha _{1}\leq \frac{1}{2}\min \left\{ \frac{\beta _{1}}{C}, \beta _{1} \right\}.
\end{gather}
This along with \eqref{bd-stat} gives rise to
\begin{gather}\label{decay-l2}
\displaystyle  \|\varphi \| _{H ^{1}}^{2}+\|\psi \| _{L ^{2}}^{2} 	\leq C \left(\|\varphi _{0}\|_{H^{1}} ^{2}+\|\psi _{0}\|_{L ^{2}}^{2}\right)	{\mathop{\mathrm{e}}}^{-\alpha _{1} t}.
\end{gather}
We thus finish the  proof of Lemma \ref{lem-smallness-assum-clos}.
\end{proof}
In what follows, we derive some higher-order estimates for the solution.
\begin{lemma}\label{lem-bd-closure}
Let $ (\varphi,\psi)\in X(0,T) $ be a solution to the initial-boundary value problem \eqref{vfi-psi-eq}--\eqref{bd-con-refor} satisfying \eqref{aprori-assum} and assume the conditions of Lemma \ref{lem-l2-esti} hold. Then it holds for any $ t \in [0,T] $ that
\begin{gather}\label{bd-assup-closure}
\displaystyle  \|\varphi _{xx}\|_{L ^{2}} + \|\psi _{x}\|_{H ^{1}} \leq \frac{3}{2}\sigma ,
\end{gather}
and that
\begin{align}\label{hig-multi}
\displaystyle \int _{0}^{t}\left( \|\varphi _{x}\|_{H ^{1}}^{2}+\|\psi _{x}\|_{H ^{1}}^{2}+\|\varphi _{\tau}\|_{H ^{1}}^{2}+\|\psi _{\tau}\|_{H ^{1}}^{2} \right)\mathrm{d}\tau \leq C \left( \|\varphi _{0}\|_{H ^{2}}^{2}+\|\psi _{0}\|_{H ^{2}}^{2} \right),
\end{align}
provided that $ \delta $ and $ \|\varphi _{0}\|_{H ^{1}}^{2}+\|\psi _{0}\|_{L ^{2}}^{2} $ are suitably small, where $ \sigma $ is given by \eqref{M1-defi}.
\end{lemma}
\begin{proof}
Let us begin with the estimate on $ \psi _{x} $. Multiplying the second equation in \eqref{vfi-psi-eq} by $\psi _{t} $ followed by an integration with respect to $ x $, we get
\begin{align}\label{psix-esti}
& \displaystyle  \frac{1}{2}\frac{\mathrm{d}}{\mathrm{d}t}\int _{\mathcal{I} }\left( D  \psi _{x}^{2}+ \bar{u}\psi ^{2} \right) \mathrm{d}x+
\int _{\mathcal{I} }\psi _{t}^{2} \mathrm{d}x =- \int _{\mathcal{I} }\psi _{t} \varphi _{x} \psi\mathrm{d}x - \int _{\mathcal{I} } \bar{v}\varphi _{x} \psi _{t} \mathrm{d}x,
 \end{align}
 where, thanks to \eqref{sobolve-l-inty-zero}, \eqref{aprori-assum} and the Cauchy-Schwarz inequality, the terms on the right hand side can be estimated as follows:
  \begin{align*}
  \displaystyle  - \int _{\mathcal{I} }\psi _{t} \varphi _{x} \psi\mathrm{d}x & \leq C \|\psi\|_{L ^{\infty}}\|\psi _{t}\|_{L ^{2}}\|\varphi _{x}\|_{L ^{2}} \leq C \|\psi _{x}\|_{L ^{2}}\|\psi _{
  t}\|_{L ^{2}}\|\varphi _{x}\|_{L ^{2}}\leq C \delta \left(  \|\psi _{t}\|_{L ^{2}} ^{2}+\|\psi _{x}\|_{L ^{2}} ^{2}\right),
   \nonumber \\
   \displaystyle   - \int _{\mathcal{I} } \bar{v}\varphi _{x} \psi _{t} \mathrm{d}x & \leq \|\bar{v}\|_{L ^{\infty}} \|\varphi _{x}\|_{L ^{2}}\|\psi _{t}\|_{L ^{2}}\leq \varepsilon \|\psi _{t}\|_{L ^{2}}+C _{ \varepsilon}\|\varphi _{x}\|_{L ^{2}}^{2}
  \end{align*}
 for any $ \varepsilon >0$. We thus update \eqref{psix-esti}, after taking $ \varepsilon $ and $ \delta $ suitably small, as
 \begin{align*}
 \displaystyle  \frac{\mathrm{d}}{\mathrm{d}t}\int _{\mathcal{I} }\left( D  \psi _{x}^{2}+ \bar{u}\psi ^{2} \right) \mathrm{d}x+
\int _{\mathcal{I} }\psi _{t}^{2} \mathrm{d}x \leq C \left( \|\varphi _{x}\|_{L ^{2}}^{2}+ \|\psi
 _{x}\|_{L ^{2}}^{2} \right) .
 \end{align*}
This along with \eqref{bd-stat} and \eqref{con-1-l2} implies for any $ t \in [0,T] $ that
\begin{gather}\label{psi-x-es}
\displaystyle \|\psi(\cdot,t) \|_{H^{1}}^{2}+ \int _{0}^{t}\|\psi _{\tau}\|_{L ^{2}}^{2}\mathrm{d}\tau \leq C\|\psi _{0x}\|_{L ^{2}}^{2}+C \left(\|\varphi _{0}\|_{L ^{2}}^{2}+\|\psi _{0}\|_{L ^{2}}^{2} \right).
\end{gather}

To complete the proof, it now remains to derive the $ H ^{2} $-estimates. Differentiating \eqref{vfi-psi-eq} with respect to $ t $, we get
\begin{gather}\label{per-eq-t}
\displaystyle \begin{cases}
	\displaystyle   \varphi _{tt}= \varphi _{xxt}- \bar{v}_{x} \varphi _{xt}- \bar{u} \psi _{xt}- \varphi _{xt} \psi _{x}- \varphi _{x} \psi _{xt},\\[1mm]
 \displaystyle  \psi _{tt}=D  \psi _{xxt}- \bar{u} \psi _{t}- \bar{v} \varphi_{xt}-  \varphi _{xt} \psi-  \varphi _{x} \psi _{t}.
\end{cases}
\end{gather}
Multiplying the first equation in \eqref{per-eq-t} by $  \varphi _{t} $ and the second one by $ \psi _{t} $, integrating the resulting equation over $ \mathcal{I} $, we get
 \begin{align}\label{H2-diff}
 &\displaystyle \frac{1}{2}\frac{\mathrm{d}}{\mathrm{d}t} \int _{\mathcal{I} } \left( \varphi _{t}^{2}+\psi _{t}^{2} \right)  \mathrm{d}x + \int _{\mathcal{I} } \left( \varphi _{xt}^{2}+D\psi _{xt}^{2} +  \bar{u}\psi _{t}^{2}\right)  \mathrm{d}x
  \nonumber \\
  &~\displaystyle= - \int _{\mathcal{I} } \bar{v}_{x} \varphi _{xt} \varphi _{t}\mathrm{d}x-  \int _{\mathcal{I} }\bar{u} \psi _{xt} \varphi _{t} \mathrm{d}x-  \int _{\mathcal{I} }\bar{v} \varphi _{xt} \psi _{t} \mathrm{d}x -  \int _{\mathcal{I} } \varphi _{x} \psi _{xt} \varphi _{t} \mathrm{d}x
   \nonumber \\
   &~ \displaystyle \quad  -  \int _{\mathcal{I} } \varphi _{xt} \psi  \psi _{t} \mathrm{d}x -  \int _{\mathcal{I} } \varphi _{x} \psi _{t}^{2} \mathrm{d}x-  \int _{\mathcal{I} } \varphi _{xt} \psi _{x} \varphi _{t} \mathrm{d}x  .
 \end{align}
 We now estimate the terms on the right hand side of \eqref{H2-diff}. By \eqref{bd-stat}, \eqref{stat-vx-bd} and the Cauchy-Schwarz inequality, we have
 	\begin{align}\label{h2-linear}
 	\displaystyle - \int _{\mathcal{I} } \bar{v}_{x} \varphi _{xt} \varphi _{t}\mathrm{d}x-  \int _{\mathcal{I} }\bar{u} \psi _{xt} \varphi _{t} \mathrm{d}x - \int _{\mathcal{I} } \bar{v}\varphi _{xt} \psi _{t}\mathrm{d}x\leq \varepsilon \int _{\mathcal{I}}\left(\psi _{xt}^{2}+ \varphi _{xt}^{2} \right) \mathrm{d}x+C _{\varepsilon}\int _{\mathcal{I}}\left( \varphi _{t}^{2}+\psi _{t}^{2} \right)  \mathrm{d}x
 	\end{align}
 	for any $ \varepsilon>0 $. Thanks to \eqref{aprori-assum}, \eqref{small-implica} and the Cauchy-Schwarz inequality, we derive for $ 0<\delta \leq 1 $ that
 	\begin{align}\label{h-2-nonline}
 	&\displaystyle  -  \int _{\mathcal{I} } \varphi _{x} \psi _{xt} \varphi _{t} \mathrm{d}x  -  \int _{\mathcal{I} } \varphi _{xt} \psi  \psi _{t} \mathrm{d}x -  \int _{\mathcal{I} } \varphi _{x} \psi _{t}^{2} \mathrm{d}x
 	 \nonumber \\
 	 &~\displaystyle \leq C \|\varphi _{x}\|_{L ^{\infty}}\|\psi _{xt}\|_{L  ^{2}}\|\varphi _{t}\|_{L ^{2}}+\|\psi\|_{L ^{\infty}}\|\varphi _{xt}\|_{L ^{2}}\|\psi _{t}\|_{L ^{2}}+C \|\varphi _{x}\|_{L ^{2}}\|\psi
 	  _{t}\|_{L ^{2}}^{2}
 	   \nonumber \\
 	   &~ \displaystyle \leq  C \Big( \delta+ \delta ^{\frac{1}{2}}\sigma ^{\frac{1}{2}}\Big ) \|\psi _{xt}\|_{L  ^{2}}\|\varphi _{t}\|_{L ^{2}}  + C \delta ^{\frac{1}{2}}\sigma^{\frac{1}{2}}\|\varphi _{xt}\|_{L ^{2}}\|\psi _{t}\|_{L ^{2}} + C \delta \|\psi_{t}\|_{L ^{2}}^{2}
 	   \nonumber \\
 	   &~\displaystyle \leq C \delta ^{\frac{1}{2}} \left( \|\psi _{xt}\|_{L ^{2}} ^{2}+\|\varphi _{xt}\|_{L ^{2}}^{2}\right)+C \Big(  \delta + \delta ^{\frac{1}{2}}\sigma \Big)\left(\| \varphi _{t}\|_{L ^{2}}^{2}+\|\psi _{t}\|_{L ^{2}}^{2} \right).
  	\end{align}
  	For the last term on the right hand side of \eqref{H2-diff}, we get from \eqref{sobolve-l-inty-zero} and Young's inequality that
  	\begin{align}\label{h-2-last}
  	\displaystyle  -  \int _{\mathcal{I} } \varphi _{xt} \psi _{x} \varphi _{t} \mathrm{d}x &
  	\leq C \|\varphi _{xt}\|_{L ^{2}}^{\frac{3}{2}}\|\varphi _{t}\|_{L ^{2}}^{\frac{1}{2}}\|\psi _{x}\|_{L ^{2}}
  	\leq  \varepsilon \|\varphi _{xt}\|_{L ^{2}}^{2}  +C _{\varepsilon}  \|\varphi _{t}\|_{L ^{2}}^{2}\|\psi _{x}\|_{L ^{2}}^{4}
  	\end{align}
  	for any $ \varepsilon>0 $, where we have used the fact $ \varphi _{t}(0,t)=\varphi _{t}(1,t)=0 $ due to \eqref{bd-con-refor}. Substituting \eqref{h2-linear}--\eqref{h-2-last} into \eqref{H2-diff} and then taking $ \varepsilon $ and $ \delta $ small enough, we get that
  	\begin{align}\label{vfi-t-pis-t}
  	&\displaystyle   \frac{1}{2}\frac{\mathrm{d}}{\mathrm{d}t} \int _{\mathcal{I} } \left( \varphi _{t}^{2}+\psi _{t}^{2} \right)  \mathrm{d}x + \frac{1}{2}\int _{\mathcal{I} } \left( \varphi _{xt}^{2}+D\psi _{xt}^{2} +  \bar{u}\psi _{t}^{2}\right)  \mathrm{d}x
  	 \nonumber \\
  	 & ~\displaystyle \leq C \Big(1+\delta ^{\frac{1}{2}}\sigma \Big)\int _{\mathcal{I}}\left( \varphi _{t}^{2}+\psi _{t}^{2} \right)  \mathrm{d}x+\|\varphi _{t}\|_{L ^{2}}^{2}\|\psi _{x}\|_{L ^{2}}^{4}.
  	\end{align}
  	Recalling \eqref{con-1-l2} and \eqref{vfi-x}, we have
\begin{align}\label{vfi-t-multi}
\displaystyle  \int _{\mathcal{I} }\varphi _{x}^{2} \mathrm{d}x+ \int _{0}^{t}\int _{\mathcal{I} } \varphi _{\tau}^{2} \mathrm{d}x \mathrm{d}\tau\leq \left\|\varphi _{0x}\right\|_{L ^{2}}^{2}+ C \int _{0}^{t}\|\psi _{x}\|_{L ^{2}}^{2}\mathrm{d}\tau \leq C \left(\|\varphi _{0}\|_{H ^{1}}^{2}+\|\psi _{0}\|_{L ^{2}}^{2}\right),
\end{align}
  	provided $ \delta $ is small enough and \eqref{DE-M1-const} holds. Combining \eqref{psi-x-es} with \eqref{vfi-t-pis-t} and \eqref{vfi-t-multi}, we then get
  	\begin{align}\label{vfi-t-con}
  	 &\displaystyle   \int _{\mathcal{I} } \left( \varphi _{t}^{2}+\psi _{t}^{2} \right)  \mathrm{d}x +  \int _{0}^{t}\int _{\mathcal{I} } \left( \varphi _{x \tau}^{2}+\psi _{x \tau}^{2} +  \bar{u}\psi _{\tau}^{2}\right)  \mathrm{d}x \mathrm{d}\tau \nonumber \\
  	   	&~\displaystyle \leq C \int _{\mathcal{I}}\left( \varphi _{0xx}^{2}+\varphi _{0x}^{2}+\psi _{0xx}^{2}+\psi _{0x}^{2} \right)\mathrm{d}x+\Big(\sup _{0 \leq \tau \leq t}\|\psi _{x}\|_{L ^{2}}^{4}\Big)\int _{0}^{t}\int _{\mathcal{I}}\varphi _{\tau}^{2}\mathrm{d}x\mathrm{d}\tau
  	   	 \nonumber \\
  	   	 &~\displaystyle \quad +C \Big(1+\delta ^{\frac{1}{2}}\sigma \Big)\int _{0}^{t}\int _{\mathcal{I}}\left( \varphi _{\tau}^{2}+\psi _{\tau}^{2} \right)  \mathrm{d}x \mathrm{d}\tau
  	   	 \nonumber \\
  	   	 &~\displaystyle \leq C \left( \|\varphi _{0xx}\|_{L ^{2}}^{2}+\|\psi _{0x}\|_{H^{1}}^{2} +\|\varphi _{0x}\|_{L ^{2}}^{2}\right) +
  	   	 C \Big( 1+\delta ^{\frac{1}{2}}\sigma\Big) \left(  \|\varphi _{0}\|_{L ^{2}}^{2}+\|\psi _{0}\|_{L ^{2}}^{2} \right)
  	   	 \nonumber \\
  	   	 &~ \displaystyle \quad +C \delta ^{\frac{1}{2}}\sigma\|\psi _{0x}\|_{L ^{2}}^{2}+ C \left(\|\varphi _{0}\|_{H ^{1}}^{2}+\|\psi _{0}\|_{L ^{2}}^{2}\right) \left(\|\varphi _{0}\|_{L ^{2}}^{2}+\|\psi _{0}\|_{H ^{1}}^{2}\right)^{2}
  	\end{align}
  	for any $ t \in [0,T] $, where $ \varphi _{t}\vert_{t=0}= \varphi _{0xx}-  \varphi _{0x}\bar{v}_{x}-\bar{u}\psi _{0x}- \varphi _{0x}\psi _{0x} $ and $ \psi _{t}\vert_{t=0}=D \psi _{0xx}-\bar{u}\psi _{0}- \bar{v}\varphi _{0x} -  \varphi _{0x} \psi _{0} $  from \eqref{vfi-psi-eq} have been used. From \eqref{sobole}, \eqref{bd-stat}, \eqref{stat-vx-bd}, the first equation in \eqref{vfi-psi-eq} and the Cauchy-Schwarz inequality, we have
  	\begin{align*}
  	\displaystyle  \displaystyle \int _{\mathcal{I}}\varphi _{xx}^{2}\mathrm{d}x &\leq C \int _{\mathcal{I}}\left( \varphi _{t}^{2}+\varphi _{x}^{2}+\psi _{x}^{2} \right)\mathrm{d}x+\|\varphi _{x}\|_{L ^{\infty}}^{2}\int _{\mathcal{I}}\psi _{x}^{2}\mathrm{d}x
  	 \nonumber \\
  	 &\displaystyle \leq C \int _{\mathcal{I}}\left( \varphi _{t}^{2}+\varphi _{x}^{2}+\psi _{x}^{2} \right)\mathrm{d}x+ C \left( \|\varphi _{x}\|_{L ^{2}}\|\varphi _{xx}\|_{L ^{2}}+\|\varphi _{x}\|_{L ^{2}}^{2} \right)   \int _{\mathcal{I}}\psi _{x}^{2}\mathrm{d}x
  	  \nonumber \\
  	  &\displaystyle \leq C \int _{\mathcal{I}}\left( \varphi _{t}^{2}+\varphi _{x}^{2}+\psi _{x}^{2} \right)\mathrm{d}x+ \frac{1}{2}\|\varphi _{xx}\|_{L ^{2}}^{2}+C\|\varphi _{x}\|_{L ^{2}}^{2}\left( \|\psi _{x}\|_{L ^{2}}^{2}+\|\psi _{x}\|_{L ^{2}} ^{4}\right),
  	\end{align*}
  	and thus
  	\begin{align*}
  	\displaystyle  \int _{\mathcal{I}}\varphi _{xx}^{2}\mathrm{d}x \leq C \int _{\mathcal{I}}\left( \varphi _{t}^{2}+\varphi _{x}^{2}+\psi _{x}^{2} \right)\mathrm{d}x+C\|\varphi _{x}\|_{L ^{2}}^{2}\left( \|\psi _{x}\|_{L ^{2}}^{2}+\|\psi _{x}\|_{L ^{2}} ^{4}\right).
  	\end{align*}
  	This together with \eqref{con-1-l2}, \eqref{psi-x-es}, \eqref{vfi-t-multi} and \eqref{vfi-t-con} yields that
  	\begin{align}\label{vfi-xx}
  	\displaystyle  \int _{\mathcal{I}}\varphi _{xx}^{2}\mathrm{d}x &\leq C \left( \|\varphi _{0xx}\|_{L ^{2}}^{2}+\|\psi _{0x}\|_{H^{1}}^{2} +\|\varphi _{0x}\|_{L ^{2}}^{2}\right) +
  	   	 C \Big( 1+\delta ^{\frac{1}{2}}\sigma  \Big) \left(  \|\varphi _{0}\|_{L ^{2}}^{2}+\|\psi _{0}\|_{L ^{2}}^{2} \right)
  	   	 \nonumber \\
  	   	 &~ \displaystyle \quad +C \delta ^{\frac{1}{2}}\sigma\|\psi _{0x}\|_{L ^{2}}^{2}+ C \left(\|\varphi _{0}\|_{L ^{2}}^{2}+\|\psi _{0}\|_{L ^{2}}^{2}\right) \left( \|\varphi _{0}\|_{L ^{2}}^{2}+\|\psi _{0}\|_{H ^{1}}^{2}\right)^{2},
  	\end{align}
  	and that
  	\begin{align}\label{multi-vfi-xx}
  	\displaystyle \int _{0}^{t}\int _{\mathcal{I}}\varphi _{xx}^{2}\mathrm{d}x \mathrm{d}\tau  & \leq C \int _{0}^{t}\int _{\mathcal{I} }(\varphi _{\tau}^{2}+\varphi _{x}^{2}+\psi _{x}^{2}) \mathrm{d}x \mathrm{d}\tau+\sup _{0 \leq \tau \leq t}\left( \|\psi _{x}\|_{L ^{2}}^{2}+\|\psi _{x}\|_{L ^{2}}^{4} \right)\int _{0}^{t}\|\varphi _{x}\|_{L ^{2}}^{2}\mathrm{d}\tau
  	 \nonumber \\
  	 &\displaystyle \leq C \left(\|\varphi _{0}\|_{ H ^{1}}^{2}+\left\|\psi _{0}\right\|_{L ^{2}}\right)+\left( \|\varphi _{0}\|_{L ^{2}}^{2}+\|\psi _{0}\|_{L ^{2}}^{2} \right) \left( \|\psi _{0}\|_{H ^{1}}^{2}+\|\varphi _{0}\|_{L ^{2}}^{2} \right)^{2}
  	 \nonumber \\
  	 &\displaystyle\leq C \left(\|\varphi _{0}\|_{H ^{1}}^{2}+\|\psi _{0}\|_{L ^{2}}^{2}\right) \left( 1+\|\varphi _{0}\|_{L ^{2}}^{2}+\|\psi _{0}\|_{H ^{1}}^{2}\right)^{2}.
  	\end{align}
  	where \eqref{DE-M1-const} has been used. Similarly, recalling \eqref{sobolve-l-inty-zero}, \eqref{bd-con-refor}, \eqref{decay-l2}, \eqref{psi-x-es}, \eqref{vfi-t-multi}, \eqref{vfi-t-con} and the second equation in \eqref{vfi-psi-eq}, we have
  	\begin{align}\label{psi-xx}
  	\displaystyle  \int _{\mathcal{I}}\psi _{xx}^{2}\mathrm{d}x & \leq C \int _{\mathcal{I}}\left( \psi _{t}^{2}+\psi ^{2}+\varphi _{x}^{2} \right) \mathrm{d}x + \|\psi\|_{L ^{\infty}}^{2}\int _{\mathcal{I}}\varphi _{x} ^{2}\mathrm{d}x
  	 \nonumber \\
  	 &\leq C \int _{\mathcal{I}}\left( \psi _{t}^{2}+\psi ^{2}+\varphi _{x}^{2} \right) \mathrm{d}x+ \|\psi\|_{ L ^{2}}\|\psi _{x}\|_{L ^{2}}\int _{\mathcal{I}}\varphi _{x} ^{2}\mathrm{d}x
  	  \nonumber \\
  	    	 & \leq C \left( \|\varphi _{0xx}\|_{L ^{2}}^{2}+\|\psi _{0x}\|_{H^{1}}^{2} +\|\varphi _{0x}\|_{L ^{2}}^{2}\right) +
  	   	 C \Big( 1+\delta ^{\frac{1}{2}}\sigma \Big) \left(  \|\varphi _{0}\|_{L ^{2}}^{2}+\|\psi _{0}\|_{L ^{2}}^{2} \right)
  	   	 \nonumber \\
  	   	 &~ \displaystyle \quad +C \delta ^{\frac{1}{2}}\sigma\|\psi _{0x}\|_{L ^{2}}^{2}+ C \left(\|\varphi _{0}\|_{H ^{1}}^{2}+\|\psi _{0}\|_{L ^{2}}^{2}\right) \left( \|\varphi _{0}\|_{L ^{2}}^{2}+\|\psi _{0}\|_{H ^{1}}^{2}\right)^{2}
  	\end{align}
  	and
  	\begin{align}\label{muti-psi-xx}
  	\displaystyle  \int _{0}^{t}\int _{\mathcal{I}}\psi _{xx}^{2}\mathrm{d}x \mathrm{d}\tau & \leq C \int _{0}^{t}\int _{\mathcal{I} }(\psi _{\tau}^{2}+\psi ^{2}+\varphi _{x}^{2}) \mathrm{d}x \mathrm{d}\tau+ \Big(\sup _{\tau \in[0,t]}\|\psi(\cdot,\tau)\|_{H ^{1}}^{2}\Big)\int _{0}^{t}\int _{\mathcal{I} }\varphi _{x}^{2} \mathrm{d}x \mathrm{d}\tau
  	 \nonumber \\
  	  & \leq
  	C \left( \|\psi _{0}\|_{H ^{1}}^{2}+\|\varphi _{0}\|_{L ^{2}}^{2} \right)+ C \left( \|\varphi _{0}\|_{L ^{2}}^{2}+\|\psi _{0}\|_{L ^{2}} ^{2}\right)\left( \|\psi _{0}\| _{H ^{1}}^{2}+\|\varphi _{0}\|_{L ^{2}}^{2}\right)
  	 \nonumber \\
  	 &\displaystyle    	
  	\leq C \left(1+ \|\varphi _{0}\|_{L ^{2}}^{2}+\|\psi _{0}\|_{L ^{2}}^{2} \right)\left(\|\varphi _{0}\|_{L ^{2}}^{2}+\|\psi _{0}\|_{H ^{1}}^{2} \right).
  	  	\end{align}
  	  	Combining \eqref{psi-x-es}, \eqref{vfi-xx} and \eqref{psi-xx}, we arrive at
  	  	\begin{align*}
  	  	&\displaystyle  \|\varphi _{xx}(\cdot,t)\|_{L ^{2}}^{2} + \|\psi _{x}(\cdot,t)\|_{H ^{1}}^{2} + \int _{0}^{t}\int _{\mathcal{I} } \left( \varphi _{x \tau}^{2}+\psi _{x \tau}^{2} +  \bar{u}\psi _{\tau}^{2}\right)  \mathrm{d}x \mathrm{d}\tau
  	  	 \nonumber \\
  	  	   	   	  &~\displaystyle \leq C \left( \|\varphi _{0xx}\|_{L ^{2}}^{2}+\|\psi _{0x}\|_{H^{1}}^{2} \right) +
  	   	 C \Big( 1+\delta ^{\frac{1}{2}}\sigma  \Big) \left(  \|\varphi _{0}\|_{H ^{1}}^{2}+\|\psi _{0}\|_{L ^{2}}^{2} \right)
  	   	 \nonumber \\
  	   	 &~ \displaystyle \quad +C \delta ^{\frac{1}{2}}\sigma\|\psi _{0x}\|_{L ^{2}}^{2}+ C \left(\|\varphi _{0}\|_{H ^{1}}^{2}+\|\psi _{0}\|_{L ^{2}}^{2}\right) \left( \|\varphi _{0}\|_{L ^{2}}^{2}+\|\psi _{0}\|_{H ^{1}}^{2}\right)^{2}.
  	  	\end{align*}
  	  	Consequently, if we take
  	  	\begin{align}\label{M1-defi}
  	  	\displaystyle \sigma ^{2} = \max \left\{ 4C \left( \|\varphi _{0xx}\|_{L ^{2}}^{2}+\|\psi _{0x}\|_{H^{1}}^{2} \right) ,1 \right\}
  	  	\end{align}
  	  	 and set both $ \delta $ and $  \|\varphi _{0}\|_{H ^{1}}^{2}+\left\|\psi _{0}\right\|_{L ^{2}}^{2} $ small enough such that \eqref{DE-M1-const} and
  	  	 \begin{align}\label{deta-initial-data-constrian}
  	  	 \displaystyle
  	   	& C \Big( 1+\delta ^{\frac{1}{2}}\sigma \Big) \left(  \|\varphi _{0}\|_{H ^{1}}^{2}+\|\psi _{0}\|_{L ^{2}}^{2} \right)+C \delta ^{\frac{1}{2}}\sigma\|\psi _{0x}\|_{L ^{2}}^{2}
  	   	 \nonumber \\
  	   	 &~ \displaystyle \quad + C \left(\|\varphi _{0}\|_{H ^{1}}^{2}+\|\psi _{0}\|_{L ^{2}}^{2}\right) \left(\|\varphi _{0}\|_{L ^{2}}^{2}+\|\psi _{0}\|_{H ^{1}}^{2}\right)^{2} \leq 2 \sigma^{2}
  	  	 \end{align}
  	  	 are satisfied, then it holds that
  	  	 \begin{align}\label{vfi-xx-final}
  	  	 \displaystyle   \|\varphi _{xx}(\cdot,t)\|_{L ^{2}}^{2} + \|\psi _{x}(\cdot,t)\|_{H ^{1}}^{2} + \int _{0}^{t}\int _{\mathcal{I} } \left( \varphi _{x \tau}^{2}+\psi _{x \tau}^{2} +  \bar{u}\psi _{\tau}^{2}\right)  \mathrm{d}x \mathrm{d}\tau \leq \frac{9}{4}\sigma^{2}.
  	  	 \end{align}
  	  	 This gives \eqref{bd-assup-closure}. Differentiating the first equation in \eqref{vfi-psi-eq} with respect to $ x $ leads to
  	  	 \begin{align*}
  	  	 \displaystyle \varphi _{xxx}=\varphi _{xt}+\bar{v}_{xx}\varphi _{x}-\bar{v}_{x}\varphi _{x}- \bar{u}_{x}\psi _{x}-\bar{v}\psi _{xx}- \varphi _{xx}\psi _{x}-\varphi _{x}\psi _{xx},
  	  	 \end{align*}
  	  	 which in combination with \eqref{stat-problem}, \eqref{bd-stat}, \eqref{stat-vx-bd}, \eqref{con-1-l2}, \eqref{multi-vfi-xx}, \eqref{muti-psi-xx}, \eqref{vfi-xx-final} and the Sobolev inequality \eqref{sobole} yields that
  	  	 \begin{align}\label{multi-vfi-xxx}
  	  	 \displaystyle \int _{0}^{t}\int _{\mathcal{I}}\varphi _{xxx}^{2}\mathrm{d}x \mathrm{d}\tau \leq C \left(\|\varphi _{0}\|_{H ^{2}}^{2}+\|\psi _{0}\|_{H ^{2}}^{2}\right),
  	  	 \end{align}
  	  	 provided  $  \|\varphi _{0}\|_{H ^{1}}^{2}+\left\|\psi _{0}\right\|_{L ^{2}}^{2} $ is suitably small. Similarly, we utilize \eqref{stat-refor}, \eqref{bd-stat}, \eqref{stat-vx-bd}, \eqref{con-1-l2}, \eqref{vfi-xx} and \eqref{vfi-xx-final} to get
  	  	 \begin{align}\label{multi-psi-xxx}
  	  	 \displaystyle \int _{0}^{t}\int _{\mathcal{I}}\psi _{xxx}^{2}\mathrm{d}x \mathrm{d}\tau \leq C   \left(\|\varphi _{0}\|_{H ^{2}}^{2}+\|\psi _{0}\|_{H ^{2}}^{2}\right),
  	  	 \end{align}
  	  	 provided  $  \|\varphi _{0}\|_{H ^{1}}^{2}+\left\|\psi _{0}\right\|_{L ^{2}}^{2} $ is suitably small. Combining \eqref{multi-vfi-xx}, \eqref{muti-psi-xx}, \eqref{multi-vfi-xxx} and \eqref{multi-psi-xxx}, we then obtain \eqref{hig-multi} and finish the proof of Lemma \ref{lem-bd-closure}.  	
\end{proof}
       \noindent{\bf \emph{Proof of Proposition \ref{prop-aprioriesti}}.} According to Lemmas \ref{lem-stat}--\ref{lem-bd-closure}, to finish the proof of Proposition \ref{prop-aprioriesti}, it suffices to close the \emph{a priori} assumptions \eqref{aprori-assum}. To this end, we first fix $ \sigma $ by \eqref{M1-defi} and choose $ \delta $ and $  \|\varphi _{0}\|_{H ^{1}}^{2}+\left\|\psi _{0}\right\|_{L ^{2}}^{2} $ suitably small such that \eqref{deta-constrian-lem1}, \eqref{DE-M1-const} and  \eqref{deta-initial-data-constrian} hold. Then in view of \eqref{con-decay-l2} and \eqref{bd-assup-closure}, the \emph{a priori} assumption \eqref{aprori-assum} is closed provided that $  \|\varphi _{0}\|_{H ^{1}}^{2}+\left\|\psi _{0}\right\|_{L ^{2}}^{2} $ is small enough. The proof is completed.   \hfill $ \square $\\
\vspace{2mm}

       \subsection{Proof of Theorem \ref{thm-asymtoptic}} 
       \label{ssub:subsubsection_name}
       With the unique solution $ (\varphi,\psi) \in X(0,\infty) $ obtained in Proposition \ref{prop-global-refor} to the reformulated problem \eqref{vfi-psi-eq}--\eqref{bd-con-refor}, in view of \eqref{u-v-vfipsi}, we conclude that the initial-boundary value problem \eqref{1-D-model}, \eqref{bd-con-Dnon-0} admits a unique global solution $ (u,v) $ satisfying,
       \begin{gather*}
\displaystyle  u \in C([0,\infty); H ^{1})\cap L ^{2}(0,\infty;H ^{2}),\ \ v \in C([0,\infty);H ^{2})\cap L ^{2}(0,\infty;H ^{3}).
\end{gather*}
Furthermore, according to \eqref{pro-global-eti1}, we have
\begin{gather*}
\displaystyle \|(u- \bar{u}, v- \bar{v})(\cdot,t)\|_{L ^{2}}^{2} \leq C 		{\mathop{\mathrm{e}}}^{-\alpha _{1} t},\ \ \|(u _{x}- \bar{u}_{x}, \varphi _{x}-\bar{v}_{x})(\cdot,t)\|_{L ^{2}} \leq C \ \ \mbox{for any }t \geq 0,
\end{gather*}
where $ C>0 $ is independent of $ t $. This along with the Sobolev inequality \eqref{sobole} implies that
\begin{align*}
\displaystyle   \|(u- \bar{u}, v- \bar{v})(\cdot,t)\|_{L ^{\infty}} &\leq C \|(u- \bar{u}, v- \bar{v})(\cdot,t)\|_{L ^{2}}^{\frac{1}{2}}\|(u _{x}- \bar{u}_{x}, v _{x} - \bar{v} _{x})(\cdot,t)\|_{L ^{2}}^{\frac{1}{2}}
 \nonumber \\
 &\displaystyle \quad+C  \|(u- \bar{u}, v- \bar{v})(\cdot,t)\|_{L ^{2}}
  \nonumber \\
  &\leq C 		{\mathop{\mathrm{e}}}^{- \frac{\alpha _{1}}{4}t} \ \ \mbox{for any }t \geq 0.
\end{align*}
This gives \eqref{decay-in-thm} with $ \alpha= \frac{\alpha _{1}}{4} $. We thus finish the proof of Theorem \ref{thm-asymtoptic}.\hfill $ \square $

\vspace*{5mm}

\section{Asymptotic stability for the case \texorpdfstring{$ D=0 $}{D=0}}\label{sec-asymptotic-D0}
In this section, we are devoted to studying the large time behavior of solutions to the problem \eqref{1-D-model}, \eqref{bdcon-D-0} with $ D=0 $. As in the case $ D>0 $, the heart of the matter is to derive some uniform-in-time estimates on the solution.

\vspace{2mm}

\subsection{\emph{A priori} estimates} 
Now we consider the system \eqref{1-D-model} with $ D=0 $:
\begin{gather}\label{1-D-model-D-0}
 \displaystyle \begin{cases}
 \displaystyle  u _{t}=u _{xx}-(u v _{x})_{x}&\mbox{in}\ \ \mathcal{I},\\[1mm]
\displaystyle v _{t}= - uv &\mbox{in}\ \ \mathcal{I},
 \end{cases}
\end{gather}
subject to the following initial and boundary conditions
\begin{gather}\label{D-0-initial-bd}
\begin{cases}
\displaystyle (u, v)\vert _{t=0}=(u _{0},v _{0})(x),\ \ u _{0}\geq 0,\\
	\displaystyle  \displaystyle   (u _{x}-u v _{x})\vert _{x=0,1}=0.
\end{cases}
\end{gather}
We shall show that
\begin{gather*}
\displaystyle (u,v) \rightarrow (M,0)\ \ \mbox{in } L ^{\infty} \ \ \mbox{as}\ \ t \rightarrow +\infty,
\end{gather*}
where $ M= \int _{\mathcal{I}}u _{0}\mathrm{d}x $. To this end, we first reformulate the problem by defining
\begin{gather*}
 \displaystyle w(x,t)=\int _{0}^{x}(u(y,t)-M)\mathrm{d}y
 \end{gather*}
with
\begin{gather*}
\displaystyle w \vert _{t=0}= \int _{0}^{x}(u _{0} -M)\mathrm{d}y =:w _{0}(x),
\end{gather*}
which leads to the following problem in terms of $ (w,v) $:
\begin{gather}\label{D-0-reforu}
\displaystyle \begin{cases}
	\displaystyle w _{t}= w _{xx}-Mv _{x}- w _{x}v _{x},\\[1mm]
  \displaystyle v _{t}=-Mv -  w _{x} v,\\
  \displaystyle w(0,t)=w(1,t)=0,\\
\displaystyle (w,v)\vert _{t=0}=(w _{0},v _{0})(x).
\end{cases}
\end{gather}
Similar to the case $ D>0 $ studied in the previous section, one can prove the local existence of solutions to the initial-boundary value problem \eqref{D-0-reforu} in the following space
\begin{gather*}
\displaystyle X _{0}(0,T):=\left\{ (w,v)\vert\, w \in C([0,T];H _{0}^{1}\cap H ^{2}) \cap L ^{2}(0,T;H ^{3}),\, v \in C([0,T]; H ^{2})\right\}.
\end{gather*}
Precisely, we have the following local existence result.
\begin{proposition}\label{prop-local-D-0}
Assume $ w _{0}\in H _{0}^{1}\cap H ^{2} $ and $ v _{0}\in H ^{2} $ such that $ w _{0x}+M \geq 0 $ and $ v _{0} \geq 0 $. Then there exists a positive constant $  T _{\ast} $ depending on the initial data such that there exists a unique solution $ (w, v) \in X _{0}(0,T _{\ast}) $ to the problem \eqref{D-0-reforu} with $ \|w\|_{H ^{2}}^2 +\|v\|_{H ^{2}}^2 \leq 2 \left( \|w _{0}\|_{H ^{2}}^2 +\|v _{0}\|_{H ^{2}}^2  \right) $ and
\begin{gather*}
\displaystyle w _{x}+M \geq 0,\ \ v \geq 0
\end{gather*}
for any $ (x,t) \in \mathcal{I}\times [0,T _{\ast}) $.
\end{proposition}

Next, we state the global existence result for the initial-boundary value problem \eqref{D-0-reforu}, from which we can obtain the global existence and large time behavior of the solution $ (u,v) $ to the problem \eqref{1-D-model-D-0}, \eqref{D-0-initial-bd}.
\begin{proposition}\label{prpo-global-D-0}
Let  $ w _{0}\in H _{0}^{1}\cap H ^{2} $ and $ v _{0}\in H ^{2} $ such that $ w _{0x}+M \geq 0 $ and $ v _{0} \geq 0 $. Then there exists a positive constant $ \delta _{\ast} $ such that if $ \|w _{0}\|_{H ^{1}}^{2}+\|v _{0}\|_{H ^{1}}^{2}\leq \delta _{\ast} $, the unique solution of the problem \eqref{D-0-reforu} obtained in Proposition \ref{prop-local-D-0} exists globally in time. Furthermore, it holds that
\begin{gather}\label{decay-global-w-v}
 \displaystyle   \|w _{xx}(\cdot,t)\|_{L ^{2}}^{2}+\|v _{xx}(\cdot,t)\|_{L ^{2}}^{2}+ \|w(\cdot,t)\|_{H ^{1}}^{2}+\|v (\cdot,t)\|_{H ^{1}}^{2} \leq C 		{\mathop{\mathrm{e}}}^{- \alpha _{2}t}\ \ \  t \geq 0,
 \end{gather}
where $ C>0 $ is a constant independent of $ t$.

 \end{proposition}

To prove Proposition \ref{prpo-global-D-0}, by the local existence result and the standard continuation argument, we just need to establish some \emph{a priori} estimates as stated in the following proposition.
\begin{proposition}\label{prop-apriori-D-0}
For any $ T>0 $, let $ (w,v) \in X _{0}(0,T) $ be a solution to the initial-boundary value problem \eqref{D-0-reforu}. Then there exists a constant $ \hat{C}_{0}>0 $ independent of $ T $ such that if $ \|w _{0}\|_{H ^{1}}^{2}+\|v _{0}\|_{H ^{1}}^{2}\leq \hat{C} _{0} $, then the solution $ (w,v)  $ possesses the following estimates:
\begin{subequations}\label{decay-and-bded}
 \begin{gather}
 \displaystyle   \|w(\cdot,t)\|_{H ^{1}}^{2}+\|v (\cdot,t)\|_{H ^{1}}^{2} \leq C 		{\mathop{\mathrm{e}}}^{- \alpha _{2}t},\\
 \displaystyle \|w _{xx}\|_{L ^{2}}^{2}+\|v _{xx}(\cdot,t)\|_{L ^{2}}^{2}+\int _{0}^{t} (\|w\|_{H ^{3}}^{2}+\|w _{\tau}\|_{H ^{1}}^{2}+\|v _{\tau}\|_{H ^{2}}^{2})\mathrm{d}\tau \leq C
 \end{gather}
\end{subequations}
for any $ t \in [0,T] $, where $ \alpha _{2} $ is as in \eqref{decay-first}, the constant $ C>0 $ is independent of $ T $.
\end{proposition}
\begin{proof}
 The proof of Proposition \ref{prop-apriori-D-0} consists of Lemmas \ref{lem-h-1-D0}--\ref{lem-higher-D0} below.
\end{proof}

Before proceeding, we assume that the solution $ (w,v) $ to the problem \eqref{D-0-reforu} satisfy the following \emph{a priori }assumptions:
\begin{gather}\label{apriori-ass-D-0}
 \displaystyle  \|w(\cdot,t)\|_{H ^{1}}+\|v (\cdot,t)\|_{H ^{1}} \leq 2\tilde{\delta},\ \ \|w _{xx}(\cdot,t)\|_{L ^{2}}\leq 2 \tilde{\sigma} \ \ \mbox{for any } t \in [0,T],
 \end{gather}
 where $ 0<\tilde{\delta}<1 $ and $ \tilde{\sigma} \geq 1 $ are constants to be determined later. Now let us derive the estimates on the $ H ^{1} $-norm of $ (w,v) $ with \eqref{apriori-ass-D-0}.
\begin{lemma}\label{lem-h-1-D0}
For any $ T>0 $, let $ (w,v) \in X _{0}(0,T) $ be a solution to the initial-boundary value problem \eqref{D-0-reforu} satisfying \eqref{apriori-ass-D-0}. Then it holds for $ \tilde{\sigma}\geq 1 $ that
\begin{gather}\label{con-v-w-h1-esti}
\displaystyle   \|w\|_{H ^{1}}^{2}+\|v \|_{H ^{1}}^{2} \leq C (\tilde{\sigma} ^{4}\|v _{0}\|_{H ^{1}}^{2}+ \|w _{0}\|_{ H ^{1}}^{2})\ \  \mbox{for all }t \in [0,T],
\end{gather}
provided $ \tilde{\delta} $ is suitably small. Furthermore, we have the following decay estimate
\begin{gather}\label{decay-first}
 \displaystyle  \|w\|_{H ^{1}}^{2}+\|v\|_{H ^{1}}^{2} \leq C (\tilde{\sigma} ^{4}\|v _{0}\|_{H ^{1}}^{2}+ \|w _{0}\|_{ H ^{1}}^{2})		{\mathop{\mathrm{e}}}^{- \alpha _{2}t}\ \  \mbox{for all }t \in [0,T],
 \end{gather}
 where $ \alpha _{2} $ and $ C $ are positive constants independent of $ t$ and $ \tilde{\sigma} $.

\end{lemma}
\begin{proof}
We divide the proof into three steps.\\
\noindent \emph{Step 1:} Estimates on $ w $. Multiplying the first equation in \eqref{D-0-reforu} followed by an integration over $ \mathcal{I} $, we have
 \begin{align}\label{diff-0-l2}
\displaystyle \frac{1}{2}\frac{\mathrm{d}}{\mathrm{d}t}\int _{\mathcal{I}}w ^{2}\mathrm{d}x + \int _{\mathcal{I} }w _{x}^{2} \mathrm{d}x &= - \int _{\mathcal{I} }w Mv _{x} \mathrm{d}x- \int _{\mathcal{I} }w _{x}w v _{x} \mathrm{d}x.
\end{align}
With integration by parts and the Cauchy-Schwarz inequality, we get
\begin{align}\label{0esit-l2err1}
\displaystyle - \int _{\mathcal{I} }w Mv _{x} \mathrm{d}x&=\int _{\mathcal{I} }Mv w _{x} \mathrm{d}x \leq \eta \int _{\mathcal{I} }w _{x}^{2} \mathrm{d}x+C _{\eta}\int _{\mathcal{I} }v ^{2} \mathrm{d}x
 \end{align}
for any $ \eta>0 $. In view of \eqref{apriori-ass-D-0}, the Cauchy-Schwarz inequality and the Sobolev inequality \eqref{sobolve-l-inty-zero}, we derive
\begin{align}\label{0esit-l2err2}
\displaystyle - \int _{\mathcal{I} }w _{x}w v _{x} \mathrm{d}x& \leq C \|w\|_{L ^{\infty}}\|w _{x}\|_{L ^{2}}\|v _{x}\|_{L ^{2}} \leq C\tilde{\delta} \left( \|w _{x}\|_{L ^{2}}^{2}+\|v _{x}\|_{L ^{2}}^{2} \right).
\end{align}
Inserting \eqref{0esit-l2err1} and \eqref{0esit-l2err2} into \eqref{diff-0-l2}, for suitably small $ \tilde{\delta} $ and $ \eta $, it holds that
\begin{align}\label{w-l2-final}
\displaystyle \frac{\mathrm{d}}{\mathrm{d}t}\int _{\mathcal{I} }w ^{2} \mathrm{d}x+ \int _{\mathcal{I} }w _{x}^{2} \mathrm{d}x \leq C \int _{\mathcal{I} } \left( v ^{2}+v _{x}^{2}\right)  \mathrm{d}x.
\end{align}
To proceed, multiplying the first equation in \eqref{D-0-reforu} by $ w _{t} $ and then integrating the resulting equation over $ \mathcal{I} $, we get
\begin{align}\label{w-x-esti}
\displaystyle \frac{1}{2}\frac{\mathrm{d}}{\mathrm{d}t}\int _{\mathcal{I} } w _{x}^{2} \mathrm{d}x+\int _{\mathcal{I} }w _{t}^{2} \mathrm{d}x= -M\int _{\mathcal{I} }v _{x}w _{t} \mathrm{d}x- \int _{\mathcal{I} }w _{x}v _{x}w _{t} \mathrm{d}x.
\end{align}
For the first term on the right hand side of \eqref{w-x-esti}, we utilize the Cauchy-Schwarz inequality to get
\begin{align*}
\displaystyle -M\int _{\mathcal{I} }v _{x}w _{t} \mathrm{d}x \leq \eta \|w _{t}\|^{2}+C _{\eta}\|v _{x}\|_{L ^{2}}^{2}
\end{align*}
for any $ \eta>0 $. For the last term,  from \eqref{sobole} and \eqref{apriori-ass-D-0}, we have
 \begin{gather}\label{w-x-linft-small}
 \displaystyle \|w _{x}\|_{L ^{\infty}} \leq C \tilde{\sigma} ^{\frac{1}{2}}\tilde{\delta}^{\frac{1}{2}}+C \tilde{\delta} ,
 \end{gather}
 which along with the Cauchy-Schwarz inequality implies that
\begin{align*}
\displaystyle \int _{\mathcal{I} }w _{x}v _{x}w _{t} \mathrm{d}x& \leq \|w _{x}\|_{L ^{\infty}}\|w _{t}\|_{L ^{2}}\|v _{x}\|_{L ^{2}} \leq C(\tilde{\sigma}^{\frac{1}{2}}\tilde{\delta}^{\frac{1}{2}}+ \tilde{\delta} )\left( \|w _{t}\|_{L ^{2}}^{2}+\|v _{x}\|_{L ^{2}}^{2} \right).
\end{align*}
Therefore, after taking $ \eta $ suitably small~(e.g., $ \eta< \frac{1}{4} $) and choosing $ \tilde{\delta}$ small enough such that
\begin{gather}\label{constrain-2-D0}
 \displaystyle  C\Big(\tilde{\sigma}^{\frac{1}{2}}\tilde{\delta}^{\frac{1}{2}}+ \tilde{\delta} \Big) \leq \frac{1}{4},
 \end{gather}
 we update \eqref{w-x-esti} as
\begin{align*}
\displaystyle  \frac{\mathrm{d}}{\mathrm{d}t}\int _{\mathcal{I} }w _{x}^{2} \mathrm{d}x+\int _{\mathcal{I} }w _{t}^{2} \mathrm{d}x \leq C \|v _{x}\|_{L ^{2}}^{2}.
\end{align*}
This together with \eqref{w-l2-final} gives
\begin{align}\label{w-h1-final}
\displaystyle \frac{\mathrm{d}}{\mathrm{d}t}\int _{\mathcal{I} }(w ^{2}+w _{x}^{2}) \mathrm{d}x+\int _{\mathcal{I} }(w _{x}^{2}+w _{t}^{2}) \mathrm{d}x \leq C \int _{\mathcal{I} } \left( v ^{2}+v _{x}^{2}\right)  \mathrm{d}x.
\end{align}
\noindent
\emph{Step 2:} Estimates on $ v $. In view of \eqref{w-x-linft-small}, it holds that
\begin{gather}\label{vfi-M-bd}
\displaystyle \frac{M}{2} \leq w _{x}+M \leq \frac{3M}{2},
\end{gather}
provided
\begin{gather}\label{constain-1-D0}
\displaystyle   C(\tilde{\sigma} ^{\frac{1}{2}}\tilde{\delta}^{\frac{1}{2}}+\tilde{\delta} )  \leq \frac{M}{2}.
\end{gather}
Therefore we test the second equation in \eqref{D-0-reforu} against $ v $ to get
\begin{align*}
\displaystyle  \frac{1}{2}\frac{\mathrm{d}}{\mathrm{d}t}\int _{\mathcal{I} }v ^{2} \mathrm{d}x+ \frac{M}{2} \int _{\mathcal{I} }v ^{2} \mathrm{d}x \leq  \frac{1}{2}\frac{\mathrm{d}}{\mathrm{d}t}\int _{\mathcal{I} }v ^{2} \mathrm{d}x+  \int _{\mathcal{I} } (w _{x}+M)v ^{2}\mathrm{d}x =0.
\end{align*}
That is,
\begin{gather}\label{v-diff}
\displaystyle \frac{\mathrm{d}}{\mathrm{d}t}\int _{\mathcal{I} }v ^{2} \mathrm{d}x +M \int _{\mathcal{I} }v ^{2} \mathrm{d}x \leq 0.
\end{gather}
Differentiating the second equation in \eqref{D-0-reforu} with respect to $ x $ gives
\begin{align}\label{v-x-eq-D0}
\displaystyle v _{xt}=-(M+ w _{x})v _{x}- w _{xx}v.
\end{align}
Multiplying \eqref{v-x-eq-D0} by $ v _{x} $ and then integrating the resulting equation over $ \mathcal{I} $, it follows that
\begin{align}\label{v-x-diff-0--first}
\displaystyle \frac{1}{2}\frac{\mathrm{d}}{\mathrm{d}t}\int _{\mathcal{I} }v _{x}^{2} \mathrm{d}x+ \int _{\mathcal{I} }(M +w _{x})v _{x}^{2} \mathrm{d}x=-\int _{\mathcal{I} }w _{xx}v v _{x} \mathrm{d}x
\end{align}
with
\begin{align*}
\displaystyle  -\int _{\mathcal{I} }w _{xx}v v _{x} \mathrm{d}x & \leq \frac{M}{16}\|v _{x}\|_{L ^{2}}^{2}+C\int _{\mathcal{I} }w _{xx}^{2}v ^{2} \mathrm{d}x
  \leq \frac{M}{16}\|v _{x}\|_{L ^{2}}^{2}+\frac{C}{M} \|v\|_{L ^{\infty}} ^{2}\int _{\mathcal{I} }w _{xx} ^{2} \mathrm{d}x
  \nonumber \\
  & \displaystyle \leq \frac{M}{16} \|v _{x}\|_{L ^{2}}^{2}+C \frac{\tilde{\sigma} ^{2} }{M}\left( \|v _{x}\|_{L ^{2}} \|v \|_{L ^{2}} +\|v \|_{L ^{2}}^{2}\right)\nonumber\\
   & \leq \frac{M}{8}\|v _{x}\|_{L ^{2}}^{2}+C \left(\frac{\tilde{\sigma} ^{2} }{M }+\frac{\tilde{\sigma}^{4}}{M ^{2}}\right)\|v\|_{L ^{2}}^{2},
\end{align*}
where we have used \eqref{sobole}, \eqref{apriori-ass-D-0} and the Cauchy-Schwarz inequality. Furthermore, thanks to \eqref{vfi-M-bd} and the fact $\tilde{\sigma} \geq 1 $, we have from \eqref{v-x-diff-0--first} that
\begin{align}\label{v-x-diff-0}
\displaystyle  \frac{1}{2}\frac{\mathrm{d}}{\mathrm{d}t}\int _{\mathcal{I} }v _{x}^{2} \mathrm{d}x+ \frac{M}{4}\int _{\mathcal{I} }v _{x}^{2} \mathrm{d}x \leq C \tilde{\sigma} ^{4} \|v\|_{L ^{2}}^{2}.
\end{align}
Combining \eqref{v-x-diff-0} with \eqref{v-diff}, we then have
\begin{align}\label{v-h1-final}
\displaystyle \frac{\mathrm{d}}{\mathrm{d}t}\int _{\mathcal{I} }(\tilde{\sigma} ^{4}v ^{2}+v _{x}^{2}) \mathrm{d}x+\hat{c}_{1}\int _{\mathcal{I} }(v ^{2}+v _{x}^{2}) \mathrm{d}x \leq 0,
\end{align}
where $ \hat{c}_{1}>0$ is a constant which depends on $ M $ but independent of $ \tilde{\sigma} $.

\noindent \emph{Step 3:} Decay estimates. Combining \eqref{v-h1-final} with \eqref{w-h1-final} yields that
\begin{align*}
\displaystyle \frac{\mathrm{d}}{\mathrm{d}t}\int _{\mathcal{I} } (\tilde{\sigma}^{4}v ^{2}+v _{x}^{2}+w ^{2}+w _{x}^{2})\mathrm{d}x+\hat{c}_{2}\left( \|v\|_{H ^{1}}^{2}+\|w _{x}\|_{L ^{2}}^{2} +\|w _{t}\|_{L ^{2}}^{2}\right) \leq 0,
\end{align*}
where $ \hat{c}_{2} $ is a constant depending on $ M $ but independent of $ \tilde{\sigma} $. Consequently, we have
\begin{align}\label{FIRST-multi-inte}
\displaystyle  \|v\|_{H ^{1}}^{2}+\|w\|_{H ^{1}}^{2}+\int _{0}^{t}\left( \|v\|_{H ^{1}}^{2}+\|w _{x}\|_{L ^{2}}^{2} +\|w _{\tau}\|_{L ^{2}}^{2}\right)\mathrm{d}\tau \leq C (\tilde{\sigma} ^{4}\|v _{0}\|_{H ^{1}}^{2}+ \|w _{0}\|_{ H ^{1}}^{2})
\end{align}
for any $ t \in [0,T] $, where we have used the fact $ \tilde{\sigma} \geq 1 $. The estimate \eqref{con-v-w-h1-esti} is proved. To show the decay estimate \eqref{decay-first}, multiplying \eqref{v-h1-final} by $ 		{\mathop{\mathrm{e}}}^{\hat{\alpha}_{1}t} $ with $\tilde{\sigma} ^{2} \hat{\alpha} _{1} \leq \hat{c}_{1}$, we deduce that
\begin{align}\label{decay-diff-1}
\displaystyle \frac{\mathrm{d}}{\mathrm{d}t}\left\{{\mathop{\mathrm{e}}}^{\hat{\alpha}_{1}t}  \int _{\mathcal{I} }(\tilde{\sigma} ^{4}v ^{2}+v _{x}^{2}) \mathrm{d}x\right\} \leq 0,
\end{align}
which immediately yields that
\begin{align}\label{v-decay-0}
 \displaystyle   \int _{\mathcal{I} }(\tilde{\sigma} ^{4}v ^{2}+v _{x}^{2}) \mathrm{d}x \leq  {\mathop{\mathrm{e}}}^{-\hat{\alpha}_{1}t}\int _{\mathcal{I} }(\tilde{\sigma} ^{4}v _{0} ^{2}+v _{0x}^{2}) \mathrm{d}x.
 \end{align}
 Since $ w(0,t)=w(1,t) =0$, we have $ \|w\|_{L ^{2}} ^{2} \leq \tilde{C}_{1} \|w _{x}\|_{L ^{2}}^{2} $ for some constant $ \tilde{C}_{1}>0 $. Multiplying \eqref{w-h1-final} by $ 		{\mathop{\mathrm{e}}}^{\hat{\alpha}_{2}t} $ with $ \hat{\alpha}_{2} < \min\{\frac{1}{2}\tilde{C}_{1},\hat{\alpha}_{1}\} $, it follows that
 \begin{align}\label{decay-diff-2}
 \displaystyle \frac{\mathrm{d}}{\mathrm{d}t}\left\{ 		{\mathop{\mathrm{e}}}^{\hat{\alpha} _{2}t} \int _{\mathcal{I} }(w ^{2}+w _{x}^{2}) \mathrm{d}x \right\} \leq 		C{\mathop{\mathrm{e}}}^{\hat{\alpha}_{2}t} \int _{\mathcal{I} } \left( v ^{2}+v _{x}^{2}\right)  \mathrm{d}x,
 \end{align}
 and thus
 \begin{align*}
 \displaystyle   \int _{\mathcal{I} }(w ^{2}+w _{x}^{2}) \mathrm{d}x \leq  C\left( \|w _{0}\|_{H ^{1}}^{2}+ \tilde{\sigma} ^{4}\|v _{0}\|_{H ^{1}} ^{2}\right) {\mathop{\mathrm{e}}}^{-\hat{\alpha} _{2}t},
 \end{align*}
 where we have used \eqref{v-decay-0} and $ \tilde{\sigma} \geq 1 $. The proof is completed.
\end{proof}
In the next lemma, we establish estimate for $ w _{xx} $.
\begin{lemma}\label{lem-bd-w-xx-singl-D0}
Assume the conditions of Lemma \ref{lem-h-1-D0} hold. Then the solution $ (w,v) \in X _{0}(0,T) $ to the problem \eqref{D-0-reforu} satisfies
\begin{gather*}
\displaystyle  \|w _{xx}\|_{L ^{2}} \leq \frac{3}{2} \tilde{\sigma} \ \ \mbox{for any }t \in [0,T],
\end{gather*}
provided that $ \tilde{\delta} $ and $ \|v _{0}\| _{H ^{1}}+\|w _{0}\|_{ H ^{1}}$ are suitably small, where $ \tilde{\sigma} \geq1 $ is determined by \eqref{A-CONSTRIAN}.

\end{lemma}
\begin{proof}
  Differentiating the first equation in \eqref{D-0-reforu} with respect to $ t $, we have
\begin{align}\label{w-t-eq}
\displaystyle  w _{tt}= w _{xxt}- M v _{xt}- w _{xt} v_{x}- w _{x} v_{xt}.
\end{align}
Multiplying \eqref{w-t-eq} by $ w _{t} $ followed by an integration over $ \mathcal{I} $, we get
\begin{align}\label{w-t-esti}
&\displaystyle \frac{1}{2}\frac{\mathrm{d}}{\mathrm{d}t}\int _{\mathcal{I} }w _{t}^{2} \mathrm{d}x+\int _{\mathcal{I} }w _{xt}^{2} \mathrm{d}x =- M\int _{\mathcal{I} }v _{xt}w _{t}  \mathrm{d}x- \int _{\mathcal{I} }w _{xt}v _{x}w _{t} \mathrm{d}x- \int _{\mathcal{I} }w _{x}v _{xt}w _{t} \mathrm{d}x.
\end{align}
We next estimate the terms on the right hand side of \eqref{w-t-esti}. Using integration by parts and the Cauchy-Schwarz inequality, one has
\begin{align*}
\displaystyle - M\int _{\mathcal{I} }v _{xt}w _{t}  \mathrm{d}x&= M\int _{\mathcal{I} }v _{t}w _{tx} \mathrm{d}x  \leq \frac{1}{16} \int _{\mathcal{I} }w _{xt}^{2} \mathrm{d}x+ C\int _{\mathcal{I} }v _{t}^{2} \mathrm{d}x.
\end{align*}
Furthermore, recalling \eqref{vfi-M-bd} and the second equation in \eqref{D-0-reforu}, we have
\begin{align}\label{v-t-esti-0}
\displaystyle \int _{\mathcal{I} }v _{t}^{2} \mathrm{d}x& \leq C \int _{\mathcal{I} }v ^{2}(w _{x}+M)^{2} \mathrm{d}x \leq C \|v\|_{L ^{2}}^{2}.
\end{align}
It thus holds that
\begin{align}\label{w-t-first}
\displaystyle  - M\int _{\mathcal{I} }v _{xt}w _{t}  \mathrm{d}x \leq  \frac{1}{16} \int _{\mathcal{I} }w _{xt}^{2} \mathrm{d}x+C\int _{\mathcal{I} }v^{2} \mathrm{d}x,
\end{align}
provided $ C(\tilde{\sigma} ^{\frac{1}{2}}\tilde{\delta}^{\frac{1}{2}}+ \tilde{\delta} )  \leq \frac{M}{2} $. It follows from \eqref{sobolve-l-inty-zero} and the Cauchy-Schwarz inequality that
\begin{align}\label{w-t-second}
\displaystyle - \int _{\mathcal{I} }w _{xt}v _{x}w _{t} \mathrm{d}x & \leq \frac{1}{16} \int _{\mathcal{I} }w _{xt}^{2} \mathrm{d}x+C \int _{\mathcal{I} }w _{t}^{2}v _{x}^{2} \mathrm{d}x
 \leq  \frac{1}{16} \int _{\mathcal{I} }w _{xt}^{2} \mathrm{d}x+C \|w _{t}\|_{L ^{\infty}}^{2} \|v _{x}\|_{L ^{2}}^{2}
  \nonumber \\
  & \displaystyle \leq  \frac{1}{16} \int _{\mathcal{I} }w _{xt}^{2} \mathrm{d}x+ C \|w _{t}\|_{L ^{2}}\|w _{xt}\|_{L ^{2}} \|v _{x}\|_{L ^{2}}^{2}\nonumber \\
    & \leq \frac{1}{8} \int _{\mathcal{I} }w _{xt}^{2} \mathrm{d}x+ C\left\|w _{t}\right\|_{L ^{2}}^{2}\|v _{x}\|_{L ^{2}}^{4}.
\end{align}
For the last term on the right hand side of \eqref{w-t-esti}, integration by parts leads to
\begin{align}\label{last-term0}
\displaystyle \int _{\mathcal{I} }w _{x}v _{xt}w _{t} \mathrm{d}x  = - \int _{\mathcal{I} }w _{xx} v _{t}w _{t}\mathrm{d}x- \int _{\mathcal{I} }w _{x}v _{t}w _{xt} \mathrm{d}x.
\end{align}
Recalling the first equation in \eqref{D-0-reforu}, we have
\begin{align*}
\displaystyle \int _{\mathcal{I} }w _{xx}^{2} \mathrm{d}x  &\leq \int _{\mathcal{I} }w _{t}^{2} \mathrm{d}x  + M ^{2}\int _{\mathcal{I} }v _{x}^{2} \mathrm{d}x  + \int _{\mathcal{I} }w _{x}^{2}v _{x}^{2} \mathrm{d}x,
\end{align*}
where, thanks to the Sobolev inequality \eqref{sobole} and the Cauchy-Schwarz inequality, we deduce that
\begin{align*}
\displaystyle \int _{\mathcal{I} } w _{x}^{2}v _{x}^{2} \mathrm{d}x &\leq \left\| w _{x}\right\|_{L ^{\infty}}^{2}\int _{\mathcal{I} }v _{x}^{2} \mathrm{d}x
\leq C \left( \| w _{x}\|_{L ^{2}}\left\| w _{xx}\right\|_{L ^{2}}+\| w _{x}\| _{L ^{2}}^{2}\right)\|v _{x}\|_{L ^{2}}^{2}
 \nonumber \\
 &\displaystyle \leq \frac{1}{2}\| w _{xx}\|_{L ^{2}}^{2}+C \| w _{x}\|_{L ^{2}}^{2}\|v _{x}\|_{L ^{2}}^{4}+C \| w _{x}\|_{L ^{2}}^{2}\|v _{x}\|_{L ^{2}}^{2}.
\end{align*}
It then follows that
\begin{align}\label{w-xx-esti}
\displaystyle  \int _{\mathcal{I} }w _{xx}^{2} \mathrm{d}x  \leq C \|w _{t}\|_{L ^{2}}^{2} + C \| w _{x}\|_{L ^{2}}^{2}\|v _{x}\|_{L ^{2}}^{4}+C \| w _{x}\|_{L ^{2}}^{2}\|v _{x}\|_{L ^{2}}^{2}.
\end{align}
This, along with \eqref{sobolve-l-inty-zero}, \eqref{v-t-esti-0} and the Cauchy-Schwarz inequality, yields
\begin{eqnarray}\label{split-1}
\begin{aligned}
\displaystyle  - \int _{\mathcal{I} }w _{xx} v _{t}w _{t}\mathrm{d}x& \leq C \int _{\mathcal{I} }w _{xx}^{2} \mathrm{d}x+C \int _{\mathcal{I} }v _{t}^{2}w _{t}^{2} \mathrm{d}x
 \leq C \int _{\mathcal{I} }w _{xx}^{2} \mathrm{d}x+C \|w _{t}\|_{L ^{\infty}}^{2}\int _{\mathcal{I} }v _{t}^{2} \mathrm{d}x\\
     & \displaystyle \leq  C \|w _{xx}\|_{L ^{2}}^{2}+C \|w _{t}\|_{L ^{2}} \|w _{tx}\|_{L ^{2}} \|v\|_{L ^{2}}^{2}\\
     & \displaystyle \leq \frac{1}{16}\|w _{xt}\|_{L ^{2}}^{2}+C \|w _{xx}\|_{L ^{2}}^{2}+C \|w _{t}\|_{L ^{2}}^{2}\|v\|_{L ^{2}}^{4}  \\
    & \displaystyle \leq \frac{1}{16} \|w _{tx}\|_{L ^{2}}^{2}+ C \|w _{t}\|_{L ^{2}}^{2} +C\left\| w _{t}\right\|_{L ^{2}}^{2}\|v \|_{L ^{2}}^{4}+C \| w _{x}\|_{L ^{2}}^{2}\|v _{x}\|_{L ^{2}}^{2}(1+ \|v _{x}\|_{L ^{2}}^{2}).
\end{aligned}
\end{eqnarray}
It now remains to estimate the last term on the right hand side of \eqref{last-term0}. By \eqref{sobole}, \eqref{v-t-esti-0}, \eqref{w-xx-esti} and the Cauchy-Schwarz inequality, we get
\begin{align}\label{split-2}
\displaystyle  - \int _{\mathcal{I} }w _{x}v _{t}w _{xt} \mathrm{d}x & \leq \frac{1}{16}\int _{\mathcal{I} }w _{xt} ^{2} \mathrm{d}x+C\int _{\mathcal{I} }v _{t}^{2}w _{x}^{2} \mathrm{d}x\leq \frac{1}{16} \int _{\mathcal{I} }w _{xt} ^{2} \mathrm{d}x+C \|w _{x}\|_{L ^{\infty}}^{2}\int _{\mathcal{I} }v _{t}^{2} \mathrm{d}x
  \nonumber \\
  & \displaystyle \leq  \frac{1}{16} \int _{\mathcal{I} }w _{xt} ^{2} \mathrm{d}x+C\left( \|w _{x}\|_{L ^{2}}\|w _{xx}\|_{L ^{2}} +\|w _{x}\|_{L ^{2}} ^{2}\right) \|v \|_{L ^{2}}^{2}
   \nonumber \\
   &\displaystyle \leq  \frac{1}{16} \int _{\mathcal{I} }w _{xt} ^{2} \mathrm{d}x+C\|w _{xx}\|_{L ^{2}} ^{2}+C \|w _{x}\| _{L ^{2}}^{2}\|v\|_{L ^{2}} ^{4}+C\|w _{x}\|_{L ^{2}} ^{2}\|v \|_{L ^{2}}^{2}
    \nonumber \\
    & \displaystyle\leq  \frac{1}{16} \int _{\mathcal{I} }w _{xt} ^{2} \mathrm{d}x+C \|w _{t}\|_{L ^{2}}^{2}+ C \| w _{x}\|_{L ^{2}}^{2}\|v _{x}\|_{L ^{2}}^{4}+C \| w _{x}\|_{L ^{2}}^{2}\|v _{x}\|_{L ^{2}}^{2}
     \nonumber \\
     & \displaystyle \quad+C \|w _{x}\| _{L ^{2}}^{2}\|v\|_{L ^{2}} ^{4}+C \|w _{x}\|_{L ^{2}} ^{2}\|v \|_{L ^{2}}^{2} .
 \end{align}
 With \eqref{split-1} and \eqref{split-2}, we then update \eqref{last-term0} as
 \begin{align}\label{w-t-last}
 \displaystyle  \displaystyle \int _{\mathcal{I} }w _{x}v _{xt}w _{t} \mathrm{d}x  & \leq \frac{1}{8} \int _{\mathcal{I} } w _{xt}^{2} \mathrm{d}x+C \|w _{t}\|_{L ^{2}}^{2}+ C \left\| w _{t}\right\|_{L ^{2}}^{2}\|v \|_{L ^{2}}^{4}+ C \| w _{x}\|_{L ^{2}}^{2}\|v _{x}\|_{L ^{2}}^{4}
  \nonumber \\
  & \displaystyle \quad +C  \| w _{x}\|_{L ^{2}}^{2}\|v _{x}\|_{L ^{2}}^{2}+C \|w _{x}\| _{L ^{2}}^{2}\|v\|_{L ^{2}} ^{4}+C \|w _{x}\|_{L ^{2}} ^{2}\|v \|_{L ^{2}}^{2}
 \end{align}
 for any $ \eta>0 $. Combining \eqref{w-t-esti}, \eqref{w-t-first}, \eqref{w-t-second} and \eqref{w-t-last}, we arrive at
 \begin{align*}
 \displaystyle  \frac{\mathrm{d}}{\mathrm{d}t} \int _{\mathcal{I} }w _{t}^{2} \mathrm{d}x+\int _{\mathcal{I} }w _{xt}^{2} \mathrm{d}x &\leq C \left\|v _{t}\right\|_{L ^{2}}^{2}+C \left\|w _{t}\right\|_{L ^{2}}^{2}(1+\|v\|_{H ^{1}}^{4}) +C  \| w _{x}\|_{L ^{2}}^{2}\|v\|_{H^{1}}^{2}+ C  \| w _{x}\|_{L ^{2}}^{2}\|v \|_{H ^{1}}^{4}.
 \end{align*}
 Integrating the above inequality over $ [0,t] $ for any $ t \in [0,T] $, thanks to \eqref{con-v-w-h1-esti}, \eqref{v-t-esti-0} and $ w _{t}\vert_{t=0}= w _{0xx}-Mv _{0x}- w _{0x}v _{0x} $ from $ \eqref{D-0-reforu}_{1} $, one can show that
 \begin{align}\label{w-t-final-esti}
  &\displaystyle  \int _{\mathcal{I} }w _{ t}^{2} \mathrm{d}x+ \int _{0}^{t}\int _{\mathcal{I} }w _{x \tau}^{2} \mathrm{d}x \mathrm{d}\tau
   \nonumber \\
   &~\displaystyle \leq C \|w _{0xx}\|_{L ^{2}}^{2}+\|(w _{0x}+M)v _{0x}\|_{L ^{2}}^{2}+C \int _{0}^{t}\|v _{\tau}\|_{L ^{2}}^{2}\mathrm{d}\tau
    \nonumber \\
    &~\displaystyle \quad +\sup _{\tau \in [0,t]}\|v \|_{H ^{1}}^{4}\int _{0}^{t}\left( \|w _{\tau}\|_{L ^{2}}^{2}+\|w _{x}\| _{L ^{2}}^{2}\right) \mathrm{d}\tau + \sup _{\tau \in [0,t]}\|v \|_{H ^{1}}^{2}\int _{0}^{t} \|w _{x}\|_{L ^{2}}^{2}\mathrm{d}\tau
      \nonumber \\
      &~\displaystyle \leq C \|w _{0xx}\|_{L ^{2}}^{2}+C \Big(\tilde{\sigma} ^{2}\|v _{0}\|_{H ^{1}}^{2}+\|w _{0}\|_{ H ^{1}}^{2}\Big)\Big(1+\tilde{\sigma} ^{4}\|v _{0}\|_{H ^{1}}^{4}+\|w _{0}\|_{ H ^{1}}^{4}\Big).
  \end{align}
  Recalling \eqref{con-v-w-h1-esti} and \eqref{w-xx-esti}, we then get
  \begin{align}\label{w-xx-sing-esti}
  \displaystyle \int _{\mathcal{I} }w _{xx}^{2} \mathrm{d}x  &\leq C \|w _{t}\|_{L ^{2}}^{2}+C \| w _{x}\|_{L ^{2}}^{2}\|v _{x}\|_{L ^{2}}^{4}+C \| w _{x}\|_{L ^{2}}^{2}\|v _{x}\|_{L ^{2}}^{2}
   \nonumber \\
   & \displaystyle \leq C \|w _{0xx}\|_{L ^{2}}^{2}+C \Big(\tilde{\sigma} ^{2}\|v _{0}\|_{H ^{1}}^{2}+\|w _{0}\|_{ H ^{1}}^{2}\Big)\Big(1+\tilde{\sigma} ^{4}\|v _{0}\|_{H ^{1}}^{4}+\|w _{0}\|_{ H ^{1}}^{4}\Big),
  \end{align}
  where we have used \eqref{con-v-w-h1-esti}. Consequently, we obtain
\begin{align*}
  \displaystyle  \int _{\mathcal{I} }w _{xx}^{2} \mathrm{d}x  \leq \frac{9}{4}\tilde{\sigma} ^{2} \ \ \mbox{for any } t \in [0,T],
 \end{align*}
 provided
 \begin{gather}\label{A-CONSTRIAN}
 \displaystyle  C \|w _{0xx}\|_{L ^{2}}^{2}+C \Big(\tilde{\sigma} ^{2}\|v _{0}\|_{H ^{1}}^{2}+\|w _{0}\|_{ H ^{1}}^{2}\Big)\Big(1+\tilde{\sigma} ^{4}\|v _{0}\|_{H ^{1}}^{4}+\|w _{0}\|_{ H ^{1}}^{4}\Big) \leq \frac{9}{4}\tilde{\sigma} ^{2} .
 \end{gather}
 It should be pointed out the constraint \eqref{A-CONSTRIAN} on $ \tilde{\sigma} $ is reachable. Indeed, if we fix the constant $ \tilde{\sigma} \in (\max\{1, 2 \sqrt{C} \|w _{0xx}\|_{L ^{2}}\},+\infty)  $, then \eqref{A-CONSTRIAN} is automatically satisfied provided $ \|v _{0}\|_{ H ^{1}} +\|w _{0}\| _{H ^{1}}$ is suitably small. The proof is complete.
\end{proof}

\begin{remark}
According to \eqref{w-xx-sing-esti}, we can fix the constant $ \tilde{\sigma} \in (\max\{1, 2 \sqrt{C} \|w _{0xx}\|_{L ^{2}}\},+\infty) $. Furthermore, let the constant $ \tilde{\delta} $ suitably small such that \eqref{constrain-2-D0} and \eqref{constain-1-D0} hold. Then by \eqref{con-v-w-h1-esti} and Lemma \ref{lem-bd-w-xx-singl-D0}, if $ \|v _{0}\|_{ H ^{1}} +\|w _{0}\| _{H ^{1}}$ is small enough, the a priori assumption \eqref{apriori-ass-D-0} is closed.
\end{remark}

Now we have closed the \emph{a priori} assumptions in \eqref{apriori-ass-D-0} and proved most of the estimates in \eqref{decay-and-bded}. To guarantee the global existence of the solution $ (w,v) $, we need to derive some more higher-order estimates (i.e., the rest of the estimates in \eqref{decay-and-bded}), and ultimately end the proof of Proposition \ref{prop-apriori-D-0}.
\begin{lemma}\label{lem-higher-D0}
Under the conditions of Lemmas \ref{lem-h-1-D0}--\ref{lem-bd-w-xx-singl-D0}, we get for any $ t \in [0,T] $ that
 \begin{gather*}
 \displaystyle \|v _{xx}(\cdot,t)\|_{ L ^{2}}^{2}+\int _{0}^{t} \left( \|w _{xxx}\|_{L ^{2}}^{2}+\|v _{\tau}\| _{H ^{2}}^{2}\right)\mathrm{d}\tau \leq C,
 \end{gather*}
 where $ C>0 $ is a constant independent of $ t $.
\end{lemma}
\begin{proof}
Differentiating the first equation in \eqref{D-0-reforu} with respect to $ x $, we get
\begin{align*}
\displaystyle w _{xxx}=w _{xt}+Mv _{xx}+w _{xx}v _{x}+w _{x}  v_{xx}.
\end{align*}
We thus derive, thanks to \eqref{sobole}, \eqref{con-v-w-h1-esti}, \eqref{w-t-final-esti} and \eqref{w-xx-sing-esti}, that
\begin{align}\label{w-xxx-multi-esti}
\displaystyle \int _{0}^{t}\| w _{xxx}\|_{L ^{2}}^{2}\mathrm{d}\tau& \leq \int _{0}^{t}\| w _{x \tau}\|_{L ^{2}}^{2}\mathrm{d}\tau + \int _{0}^{t}\|v _{x}\|_{L ^{\infty}}^{2}\| w _{xx}\|_{L ^{2}}^{2}\mathrm{d}\tau  + M \int _{0}^{t}\|v _{xx}\|_{L ^{2}}^{2}\mathrm{d}\tau
 \nonumber \\
 & \displaystyle \quad
 + \int _{0}^{t}\| w _{x}\|_{L ^{\infty}}^{2}\|v  _{xx}\|_{L ^{2}}^{2}\mathrm{d}\tau
  \leq  C+ C\int _{0}^{t}\|v _{xx}\|_{L ^{2}}^{2}\mathrm{d}\tau
\end{align}
for any $ t \in [0,T] $, where the constant $ C>0 $ is independent of $ t $. From the second equation in \eqref{D-0-reforu}, we have
 \begin{gather}
  \displaystyle v(x,t) =v _{0}{\mathop{\mathrm{e}}}^{-\int _{0}^{t}(w _{x}+M)\mathrm{d}\tau}, \label{v-formula}\\
  \displaystyle v _{txx}=-(w _{x}+M)v _{xx}-w _{xxx}v-2w _{xx}v _{x}. \label{v-xx-eq}
  \end{gather}
  Testing \eqref{v-xx-eq} against $ v _{xx} $, and then integrating the resulting equation over $ (0,t) $ for any $ t \in (0,T] $, we get
  \begin{align}\label{int-inequal}
  \displaystyle\int _{\mathcal{I}} v _{xx}^{2}(\cdot,t) \mathrm{d}x + \int _{0}^{t}\int _{\mathcal{I}} (w _{x}+M)v _{xx}^{2} \mathrm{d}x \mathrm{d}\tau=- \int _{0}^{t}\int _{\mathcal{I}} w _{xxx}v v _{xx} \mathrm{d}x \mathrm{d}\tau-  2\int _{0}^{t}\int _{\mathcal{I}} w _{xx}v _{x}v _{xx} \mathrm{d}x \mathrm{d}\tau,
  \end{align}
  where, due to \eqref{vfi-M-bd}, \eqref{w-xxx-multi-esti} and \eqref{v-formula}, it holds that
  \begin{align}\label{v-xx-firs}
  \displaystyle  \int _{0}^{t}\int _{\mathcal{I}} w _{xxx}v v _{xx} \mathrm{d}x \mathrm{d}\tau & \leq C \int _{0}^{t}\|v\|_{L ^{\infty}} \|w _{xxx}\|_{L ^{2}}\|v _{xx}\|_{L ^{2}} \mathrm{d}\tau
 \leq C 		\int _{0}^{t}{\mathop{\mathrm{e}}}^{- \frac{M}{2}\tau} \|w _{xxx}\|_{L ^{2}}\|v _{xx}\|_{L ^{2}} \mathrm{d}\tau
   \nonumber \\
   & \displaystyle\leq \frac{M}{8} \int _{0}^{t}\int _{\mathcal{I}}v _{xx}^{2} \mathrm{d}x \mathrm{d}\tau+ \int _{0}^{t}{\mathop{\mathrm{e}}}^{- \frac{M}{2}\tau}\|v _{xx}\|_{L ^{2}}^{2} \mathrm{d}\tau+C
  \end{align}
  for some constant $ C>0 $ independent of $ t $. For the last term on the right hand side of \eqref{int-inequal}, by \eqref{sobole}, \eqref{con-v-w-h1-esti} and the Cauchy-Schwarz inequality, we have
  \begin{align}\label{v-xx-sec}
   \displaystyle  -  2\int _{0}^{t}\int _{\mathcal{I}} w _{xx}v _{x}v _{xx} \mathrm{d}x \mathrm{d}\tau &\leq \frac{M}{8} \int _{0}^{t}\int _{\mathcal{I}}v _{xx}^{2} \mathrm{d}x \mathrm{d}\tau + \int _{0}^{t}\|v _{x}\|_{L ^{\infty}}^{2}\int _{\mathcal{I}} w _{xx}^{2} \mathrm{d}x \mathrm{d}\tau
    \nonumber \\
    & \displaystyle \leq \frac{M}{8} \int _{0}^{t}\int _{\mathcal{I}}v _{xx}^{2} \mathrm{d}x \mathrm{d}\tau+ \int _{0}^{t}(\|v _{x}\|_{L ^{2}}^{2}+\|v _{xx}\|_{L ^{2}}^{2})\|w _{xx}\|_{L ^{2}}^{2}\mathrm{d} \tau      \nonumber \\
     & \displaystyle \leq  \frac{M}{8} \int _{0}^{t}\int _{\mathcal{I}}v _{xx}^{2} \mathrm{d}x \mathrm{d}\tau+C \int _{0}^{t}\|w _{xx}\|_{L ^{2}}^{2}(1+\|v _{xx}\|_{L ^{2}}^{2})\mathrm{d}\tau.
   \end{align}
   Inserting \eqref{v-xx-firs} and \eqref{v-xx-sec} into \eqref{int-inequal}, by \eqref{vfi-M-bd}, it follows that
   \begin{align}\label{v-xx-final-inte-ineq}
   &\displaystyle  \int _{\mathcal{I}} v _{xx}^{2}(\cdot,t) \mathrm{d}x + \frac{M}{4} \int _{0}^{t}\int _{\mathcal{I}}v _{xx}^{2} \mathrm{d}x \mathrm{d}\tau
    \nonumber \\
    & \displaystyle~ \leq C \int _{0}^{t}\left( {\mathop{\mathrm{e}}}^{- \frac{M}{2}\tau}+\|w _{xx}\|_{L ^{2}}^{2} \right)\|v _{xx}\|_{L ^{2}}^{2}\mathrm{d}\tau+C\int _{0}^{t}\|w _{xx}\|_{L ^{2}}^{2}\mathrm{d}\tau+C
   \end{align}
   for any $ t \in [0,T] $.  Furthermore, by \eqref{FIRST-multi-inte} and \eqref{w-xx-esti}, one can show for any $ t \in [0,T] $ that
\begin{align}\label{w-xx-multi}
\displaystyle \int _{0}^{t}\|w _{xx}\|_{L ^{2}}^{2}\mathrm{d}\tau & \leq C \int _{0}^{t}\|w _{\tau}\|_{L ^{2}}^{2}\mathrm{d}\tau +\sup _{t \in [0,T]}\left( \|v _{x}\|_{L ^{2}}^{2}+\|v _{x}\|_{L ^{2}}^{4} \right) \int _{0}^{t}\|w _{x}\|_{L ^{2}}^{2}\mathrm{d}\tau \leq C,
\end{align}
where the constant $  C>0$ depends on $\|v _{0}\|_{H ^{1}}  $ and $  \|w _{0}\|_{H ^{1}} $. This along with \eqref{v-xx-final-inte-ineq} and the Gronwall inequality implies that
\begin{align}\label{v-xx-singl}
\displaystyle  \int _{\mathcal{I}} v _{xx}^{2}(\cdot,t) \mathrm{d}x + \frac{M}{4} \int _{0}^{t}\int _{\mathcal{I}}v _{xx}^{2} \mathrm{d}x \mathrm{d}\tau \leq C.
\end{align}
Combining \eqref{v-xx-singl} with \eqref{w-xxx-multi-esti} further yields that
\begin{gather}\label{w-xxx}
\displaystyle  \int _{0}^{t}\| w _{xxx}\|_{L ^{2}}^{2}\mathrm{d}\tau \leq C
\end{gather}
for any $ t \in [0,T] $, where the constant $ C>0 $ is independent of $ t $. Finally, recalling the second equation in \eqref{D-0-reforu}, we get by virtue of \eqref{sobole}, \eqref{con-v-w-h1-esti}, \eqref{w-xx-multi}--\eqref{w-xxx} that
\begin{gather*}
 \displaystyle \int _{0}^{t}\|v _{\tau}(\cdot,\tau)\|_{ H^{2}}^{2}\mathrm{d}\tau \leq C
 \end{gather*}
 for any $ t \in [0,T] $. We thus finish the proof of Lemma \ref{lem-higher-D0}.
\end{proof}

\vspace{2mm}

 \subsection{Proof of Theorem \ref{thm-stabiliy-D0}}
With the global existence result on the initial-boundary value problem \eqref{D-0-reforu} and the decay estimates in \eqref{decay-global-w-v} for $ (w,v) $ at hand, by the same process as in the analysis for the case $ D>0 $ in the previous section, one can easily prove the global existence as well as decay estimates of solutions to the original problem \eqref{1-D-model}, \eqref{bdcon-D-0}. Therefore we omit the details here for brevity. The proof of Theorem \ref{thm-stabiliy-D0} is completed. \hfill $ \square $






 \section*{Acknowledgement} 
G. Hong is partially supported from the CAS AMSS-POLYU Joint Laboratory of Applied Mathematics postdoctoral fellowship scheme.
 Z.A. Wang was supported in part by the Hong Kong Research Grant Council General Research Fund No. PolyU 153031/17P (Q62H) and internal grant No. ZZHY from HKPU.





\end{document}